\documentclass{amsart}
\pdfoutput=1

\usepackage{amssymb}
\usepackage{microtype}
\usepackage{mathtools}
\usepackage{enumitem}
\usepackage[numbers]{natbib}

\usepackage{etoolbox}
\makeatletter
\patchcmd{\NAT@test}{\else \NAT@nm }{\else \NAT@nmfmt{\NAT@nm}}{}{}
\makeatother

\usepackage{hyperref}
\hypersetup{
  pdftitle={Quadratic projectable Runge--Kutta methods},
  pdfauthor={Ari Stern and Milo Viviani},
  pdfsubject={MSC 2020: 37M15},
  bookmarksopen=false
}

\usepackage[capitalize,nameinlink]{cleveref}

\theoremstyle{plain}
\newtheorem{theorem}{Theorem}[section]
\newtheorem{proposition}[theorem]{Proposition}
\newtheorem{lemma}[theorem]{Lemma}
\newtheorem{corollary}[theorem]{Corollary}

\theoremstyle{definition}
\newtheorem{definition}[theorem]{Definition}
\newtheorem{example}[theorem]{Example}

\theoremstyle{remark}
\newtheorem{remark}[theorem]{Remark}

\AddToHook{env/proposition/begin}{\crefalias{theorem}{proposition}}
\AddToHook{env/lemma/begin}{\crefalias{theorem}{lemma}}
\AddToHook{env/corollary/begin}{\crefalias{theorem}{corollary}}
\AddToHook{env/definition/begin}{\crefalias{theorem}{definition}}
\AddToHook{env/example/begin}{\crefalias{theorem}{example}}
\AddToHook{env/remark/begin}{\crefalias{theorem}{remark}}

\numberwithin{equation}{section}

\begin{document}

\title{Quadratic projectable Runge--Kutta methods}

\author{Ari Stern}
\address{Washington University in St.~Louis}
\email{stern@wustl.edu}

\author{Milo Viviani}
\address{Scuola Normale Superiore di Pisa}
\email{milo.viviani@sns.it}


\begin{abstract}
  Runge--Kutta methods are affine equivariant: applying a method before or after an affine change of variables yields the same numerical trajectory. However, for some applications, one would like to perform numerical integration after a \emph{quadratic} change of variables. For example, in Lie--Poisson reduction, a quadratic transformation reduces the number of variables in a Hamiltonian system, yielding a more efficient representation of the dynamics. Unfortunately, directly applying a symplectic Runge--Kutta method to the reduced system generally does not preserve its Hamiltonian structure, so many proposed techniques require computing numerical trajectories of the original, unreduced system.

  In this paper, we study when a Runge--Kutta method in the original variables descends to a numerical integrator expressible entirely in terms of the quadratically transformed variables. In particular, we show that symplectic diagonally implicit Runge--Kutta (SyDIRK) methods, applied to a quadratic projectable vector field, are precisely the Runge--Kutta methods that descend to a method (generally not of Runge--Kutta type) in the projected variables. We illustrate our results with several examples in both conservative and non-conservative dynamics.
\end{abstract}

\maketitle

\section{Introduction}

In computational mathematics, Runge--Kutta methods are among the most
widely used classes of methods for numerical integration of ordinary
differential equations. Interestingly, many of the most popular
general-purpose Runge--Kutta methods have serious drawbacks when
applied to the dynamics that arise in physical systems: for example,
they may introduce artificial numerical dissipation into conservative
systems. This has inspired a substantial line of research into
\emph{structure-preserving numerical integrators} that better capture
the dynamics of such systems. We refer the reader to the excellent
survey text by \citet*{HaLuWa2006}.

Beginning with seminal work in the late 1980s, \emph{symplectic}
Runge--Kutta methods that preserve the structure of canonical
Hamiltonian systems were characterized and found to be deeply related
to those that preserve \emph{quadratic invariants}
\citep{Cooper1987,Sa1988,Lasagni1988,BoSc1994}. Recently,
\citet{McSt2024} showed that such methods have important
structure-preserving properties that manifest in the evolution of
arbitrary quadratic observables, even those that are \emph{not}
invariant, with consequences for conservative and non-conservative
systems alike.

The present paper is motivated by the question of when and whether it
is possible to compute the numerical evolution of quadratic
observables \emph{without} computing a full numerical trajectory for
the original system. By answering this question, we also obtain a
general construction for \emph{Lie--Poisson integrators} coinciding
with the quadratic projection of a symplectic Runge--Kutta method onto
a lower-dimensional space, allowing for more efficient computation of
the corresponding Hamiltonian dynamics. We also show that such methods
may be extended to non-conservative systems, including the 2D
Navier--Stokes equations with dissipation arising from nonzero
viscosity.

\bigskip

Consider a system of ordinary differential equations
$ \dot{y} = f (y) $ on a vector space $Y$. If
$ F \colon Y \rightarrow Z $ is $ C ^1 $, then the chain rule implies
that the observable $ z = F (y) $ evolves according to
$ \dot{z} = F ^\prime (y) f (y) $. The vector field $f$ is said to be
\emph{$F$-projectable} if there exists some vector field $g$ on $Z$
such that $ F ^\prime (y) f (y) = g \bigl( F (y) \bigr) $ for all
$y \in Y$, i.e.,
$f$ and $g$ are \emph{$F$-related}.
In this case, $F$ maps solutions of $ \dot{y} = f (y) $ to solutions
of $ \dot{z} = g (z) $.

We ask when this is true for \emph{numerical} solutions when $F$ is
quadratic. Specifically, suppose we integrate $ \dot{y} = f (y) $
numerically using an $s$-stage Runge--Kutta method with time-step size
$h$,
\begin{subequations}
  \label{e:rk}\noeqref{e:rk_stages,e:rk_step}
  \begin{align}
    Y _i &= y _0 + h \sum _{ j = 1 } ^s a _{ i j } f ( Y _j ) , \label{e:rk_stages}\\
    y _1 &= y _0 + h \sum _{ i = 1 } ^s b _i f ( Y _i ) ,\label{e:rk_step}
  \end{align}
\end{subequations}
where $ a _{ i j } $ and $ b _i $ are given coefficients defining the
method and $ Y _i $ are the internal stages. When does $F$ map this to
a numerical solution of $ \dot{z} = g (z) $ by some method on $Z$? We
would like this method to be defined in terms of $z$ alone, without
needing to compute a trajectory and internal stages for $y$.  In that
case, we say the Runge--Kutta method for $ \dot{y} = f (y) $
\emph{descends} to one for $ \dot{z} = g (z) $.

The primary motivation for this question comes from \emph{Lie--Poisson
  reduction} of Hamiltonian systems, arising in applications ranging
from rigid body mechanics to (magneto)hydrodynamics. In this setting,
$F$ is a quadratic \emph{momentum map} projecting dynamics from a
higher-dimensional symplectic space $Y$ onto a lower-dimensional space
$Z$ equipped with a Lie--Poisson (rather than symplectic)
structure. \emph{Collective symplectic integrators}
\citep{McMoVe2014c,McMoVe2014e} apply a symplectic Runge--Kutta method
to the higher-dimensional system $ \dot{y} = f (y) $, then project
along $F$ to obtain a numerical solution to $ \dot{z} = g (z) $ that
preserves the Lie--Poisson structure. Here, our question asks: When
can the projected solution be computed directly on $Z$, saving the
extra computational expense associated with first computing a
higher-dimensional trajectory on $Y$?

Some recent results may be understood as answering special cases of
this question. \Citet{DaSLes2022} proved that the isospectral
symplectic Runge--Kutta (IsoSyRK) schemes of \citet{ModViv2020}
descend to Lie--Poisson integrators on quadratic Lie algebras, with
intermediate stages also in the Lie algebra, if the underlying
Runge--Kutta scheme is a symplectic diagonally implicit Runge--Kutta
(SyDIRK) method. This is in line with previous results of
\citet{Viv2020}, in which the implicit midpoint method, a particular
case of SyDIRK, is shown to descend for such problems.

The starting point for our approach is to observe that, if the
coefficients $ a _{ i j } $ and $ b _i $ are such that the
Runge--Kutta method preserves quadratic invariants \citep{Cooper1987}
(and thus is also symplectic \citep{Sa1988,Lasagni1988,BoSc1994}),
then any quadratic observable $ z = F (y) $ evolves according to
\begin{equation}
  \label{e:rk_fe}
  F ( y _1 ) = F ( y _0 ) + h \sum _{ i = 1 } ^s b _i F ^\prime ( Y _i ) f ( Y _i ) ,
\end{equation}
which \citet{McSt2024} call \emph{quadratic functional
  equivariance}. Applying the $F$-relatedness of $f$ and $g$ lets us
write this as
\begin{equation*}
  z _1 = z _0 + h \sum _{ i = 1 } ^s b _i g ( Z _i ) ,
\end{equation*}
where $ z _1 = F ( y _1 ) $, $ z _0 = F ( y _0 ) $, and
$ Z _i = F ( Y _i ) $. It therefore remains only to determine when we
can compute the internal stages $ Z _i $ entirely on $Z$, without
first computing the stages $ Y _i $. By doing so, we not only
generalize previous results on special cases of Lie--Poisson
reduction, but we also obtain methods that can be applied to
non-conservative systems, such as the semidiscretized 2D
Navier--Stokes equations.

The paper is organized as follows:
\begin{itemize}
\item \cref{s:projectable_methods} introduces \emph{quadratic
    projectable Runge--Kutta methods}, which are symplectic $s$-stage
  Runge--Kutta methods satisfying additional coefficient conditions
  allowing them to descend to $s$-stage methods on $Z$.  In addition
  to $F$-projectability of $f$, we also assume that there exists a
  vector field $\gamma$ on $Z$ satisfying
  $ F ^{ \prime \prime } \bigl( f (y) , f (y) \bigr) = \gamma \bigl( F
  (y) \bigr) $, which holds for all of our target applications.

  The main result of this section (\cref{t:qprk_iff_sydirk}) shows
  that a Runge--Kutta method is quadratic projectable if and only if
  it is a SyDIRK method. This immediately yields
  (\cref{c:sydirk_descends}) a novel algorithmic description of the
  $s$-stage methods on $Z$ to which they descend, which are generally
  not of Runge--Kutta type.

\item \cref{s:projectable_vector_fields} examines the assumptions
  introduced in the previous section, obtaining sufficient conditions
  on $F$ and $f$ to determine the needed vector fields $g$ and
  $\gamma$ explicitly. This is done in two algebraic settings, one
  involving Jordan operator algebras (\cref{t:operator_projectable}),
  the other involving Lie algebra actions with invariant bilinear
  forms (\cref{t:quadratic_lie}), with momentum maps in Hamiltonian
  mechanics being an important special case of the latter. In each
  setting, these results combine with those of the previous section to
  give algorithmic descriptions of the integrators to which SyDIRK
  methods descend (\cref{c:sydirk_alpha_beta,c:quadratic_lie_sydirk}
  respectively), with several worked examples of each.

\item Finally, \cref{s:hydrodynamics} discusses the application of
  these methods to computational hydrodynamics, following the matrix
  semidiscretization approach of \citet{Zeitlin1991,Zeitlin2004} where
  (for instance) the 2D incompressible Euler equations are
  approximated by the Lie--Poisson equations for the matrix Lie
  algebra $ \mathfrak{su}(n) $ with $n$ large. While the earlier
  theory of IsoSyRK methods \citep{ModViv2020,Viv2020} had already
  been applied to the Euler equations \citep{MoVi2020,DaSLes2022}, we
  consider two extensions that take advantage of the greater
  generality of the results in the present paper. The first is to
  Lie--Poisson integration on semidirect product matrix Lie algebras,
  with application to magnetohydrodynamics. The second is to
  arbitrary, not-necessarily-conservative matrix flows, with
  application to the 2D Navier--Stokes equations.
\end{itemize}

\section{Quadratic projectable methods and their descendents}
\label{s:projectable_methods}

We begin this section with some preliminary results on the evolution
of quadratic observables by Runge--Kutta methods, obtaining quadratic
functional equivariance under the familiar coefficient conditions of
\citet{Cooper1987}, \citet{Sa1988}, and \citet{Lasagni1988}. Then, we
introduce additional coefficient conditions defining quadratic
projectable Runge--Kutta methods, show that these conditions are
uniquely satisfied by SyDIRK methods, and describe the methods to
which they descend in the projected variables.

\subsection{Preliminaries: Runge--Kutta methods and quadratic
  observables}
\label{sec:preliminaries}

Let $Y$ and $Z$ be finite-dimensional\footnote{We make this assumption
  for simplicity, and because the applications of interest are
  generally finite-dimensional. However, many of the results may be
  adapted to infinite-dimensional Banach spaces.}  real or complex
vector spaces and $ F \colon Y \rightarrow Z $ be a quadratic map,
meaning that we can expand $F (y) $ around any $ y _0 \in Y $ as
\begin{equation*}
  F (y) = F ( y _0 ) + F ^\prime ( y _0 ) ( y - y _0 ) + \frac{1}{2} F ^{ \prime \prime } ( y - y _0, y - y _0 ) .
\end{equation*}
Note that we simply write $ F ^{ \prime \prime } ( v, w ) $ rather
than $ F ^{ \prime \prime } ( y _0 ) ( v, w ) $, since
$ F ^{ \prime \prime } $ is a constant bilinear map for quadratic $F$,
i.e., it does not depend on $ y _0 $.

The following useful lemma describes how $ F (y) \in Z $ evolves
numerically whenever a Runge--Kutta method is applied to $ y \in Y $.

\begin{lemma}
  \label{l:rk_F}
  Suppose the Runge--Kutta method \eqref{e:rk} is applied to integrate
  $ \dot{y} = f (y) $ on $Y$. If $ F \colon Y \rightarrow Z $ is
  quadratic, then
  \begin{subequations}
    \label{e:rk_F}\noeqref{e:rk_F_stages,e:rk_F_step}
    \begin{alignat}{2}
      F ( Y _i )
      &= F ( y _0 ) &&+ h \sum _{ j = 1 } ^s a _{ i j } F ^\prime ( Y _j ) f ( Y _j ) \notag \\
      &&&+ \frac{ h ^2 }{ 2 } \sum _{ j , k = 1 } ^s ( a _{ i j } a _{ i k } - a _{ i j } a _{ j k } - a _{ i k } a _{ k j } ) F ^{ \prime \prime } \bigl( f ( Y _j ) , f ( Y _k ) \bigr), \label{e:rk_F_stages} \\
      F ( y _1 )
      &= F (y _0 ) &&+ h \sum _{ i = 1 } ^s b _i F ^\prime ( Y _i ) f ( Y _i ) \notag \\
      &&&+ \frac{ h ^2 }{ 2 } \sum _{ i , j = 1 } ^s ( b _i b _j - b _i a _{ i j } - b _j a _{ j i } ) F ^{ \prime \prime } \bigl( f ( Y _i ) , f ( Y _j ) \bigr) \label{e:rk_F_step} .
    \end{alignat}
  \end{subequations}
\end{lemma}
  
\begin{proof}
  Expanding $ F ( Y _i ) $ around $ y _0 $ using \eqref{e:rk_stages} gives
  \begin{equation}
    F ( Y _i ) = F ( y _0 ) + h \sum _{ j = 1 } ^s a _{ i j } F ^\prime ( y _0 ) f ( Y _j ) + \frac{ h ^2 }{ 2 } \sum _{ j, k = 1 } ^s a _{ i j } a _{ i k } F ^{ \prime \prime } \bigl( f ( Y _j ), f (Y _k ) \bigr) . \label{e:rk_F_stages_expansion}
  \end{equation}
  For the linear terms, we can similarly expand $ F ^\prime ( Y _j ) $ around $ y _0 $ to obtain
  \begin{equation*}
    F ^\prime ( Y _j ) v = F ^\prime ( y _0 ) v + h \sum _{ k = 1 } ^s a _{ j k } F ^{ \prime \prime } \bigl( v , f ( Y _k ) \bigr) ,
  \end{equation*}
  for all $ v \in Y $. Therefore, taking $ v = f ( Y _j ) $ implies
  \begin{multline*}
      h \sum _{ j = 1 } ^s a _{ i j } F ^\prime ( y _0 ) f ( Y _j ) \\
    \begin{aligned}
      &= h \sum _{ j = 1 } ^s a _{ i j } F ^\prime ( Y _j ) f ( Y _j ) - h ^2 \sum _{ j , k = 1 } ^s a _{ i j } a _{ j k } F ^{ \prime \prime } \bigl( f ( Y _j ) , f ( Y _k ) \bigr) \\
      &= h \sum _{ j = 1 } ^s a _{ i j } F ^\prime ( Y _j ) f ( Y _j ) - \frac{ h ^2 }{ 2 } \sum _{ j , k = 1 } ^s ( a _{ i j } a _{ j k } + a _{ i k } a _{ k j } ) F ^{ \prime \prime } \bigl( f ( Y _j ) , f ( Y _k ) \bigr),
    \end{aligned}
  \end{multline*}
  where the last line symmetrizes the terms of the double sum in $j$
  and $k$. Substituting this into \eqref{e:rk_F_stages_expansion}
  yields \eqref{e:rk_F_stages}. A similar argument is used to prove
  \eqref{e:rk_F_step}.
\end{proof}

An immediate corollary of \eqref{e:rk_F_step} is a special case of the
results in \citet{McSt2024}: a Runge--Kutta method is quadratic
functionally equivariant if its coefficients satisfy the condition of
\citet{Cooper1987} to preserve quadratic invariants. As is well known,
this same coefficient condition also implies that the Runge--Kutta
method is symplectic; cf.~\citet{Sa1988,Lasagni1988}.

\begin{corollary}
  If the coefficients of a Runge--Kutta method satisfy
  \begin{equation}
    \label{e:rk_coeffs_symplectic}
    b _i b _j - b _i a _{ i j } - b _j a _{ j i } = 0 , \qquad i, j = 1 , \ldots, s ,
  \end{equation}
  then the method satisfies the quadratic functional equivariance
  condition \eqref{e:rk_fe}. In particular, if
  $ F ^\prime (y) f (y) = 0 $ for all $ y \in Y $, so that $F$ is
  invariant along $f$, then $ F ( y _1 ) = F ( y _0 ) $.
\end{corollary}

We briefly recall the relationship between the conditions
\eqref{e:rk_coeffs_symplectic} and some familiar classes of
Runge--Kutta methods. The coefficients $ a _{ i j } $ and $ b _i $ of
an $s$-stage Runge--Kutta method are often presented in the form of a
\emph{Butcher tableau}, written as
\begin{equation*}
  \begin{array}[b]{|ccc}
    a _{ 1 1 } & \cdots & a _{ 1 s } \\
    \vdots & \ddots & \vdots \\
    a _{ s s } & \cdots & a _{ s s } \\
    \hline
    b _1 & \cdots & b _s 
  \end{array} \quad .
\end{equation*}
If $ a _{ i j } = 0 $ whenever $ j \geq i $ (i.e., the matrix of
coefficients $ a _{ i j } $ is strictly lower-triangular), then
\eqref{e:rk_stages} gives an explicit formula for each stage $ Y _i $
in terms of the previous stages $ Y _1 , \ldots, Y _{ i -1 } $, and
the method is said to be an \emph{explicit Runge--Kutta} (ERK)
method. Otherwise, \eqref{e:rk_stages} defines a system of equations
that must be solved to obtain $ Y _1 , \ldots, Y _s $, and the method
is said to be an \emph{implicit Runge--Kutta} (IRK) method. An IRK
method satisfying $ a _{ i j } = 0 $ whenever $ j > i $ (i.e., the
matrix of coefficients $ a _{ i j } $ has nonzero entries on the
diagonal but not above it) is said to be a \emph{diagonally implicit
  Runge--Kutta} (DIRK) method. A DIRK method allows us to solve
\eqref{e:rk_stages} for each $ Y _i $ one-at-a-time, in
sequence---similarly to an ERK method---rather than solving a larger
system for all $s$ stages simultaneously, making this a fairly mild
form of implicitness.

There are no ERK methods satisfying \eqref{e:rk_coeffs_symplectic},
other than the trivial method $ y _1 = y _0 $, since taking $ i = j $
in \eqref{e:rk_coeffs_symplectic} implies $ b _i = 0 $ for all
$ i = 1 , \ldots, s $. However, there are many well-known families of
IRK methods satisfying \eqref{e:rk_coeffs_symplectic}, including DIRK
methods.

\begin{example}[the implicit midpoint method]
  The implicit midpoint method is the $1$-stage method with
  coefficients $ a _{ 1 1 } = \frac{1}{2} $ and $ b _1 = 1 $, directly
  seen to satisfy \eqref{e:rk_coeffs_symplectic}.
\end{example}

\begin{example}[Gauss--Legendre collocation methods]
  \label{ex:gauss}
  The $s$-stage Gauss--Legendre collocation method, whose collocation
  points correspond to the roots of the degree-$s$ Legendre
  polynomial, is an order-$ 2 s $ IRK method satisfying
  \eqref{e:rk_coeffs_symplectic}; see \citet[Section
  IV.2.1]{HaLuWa2006}. The $ s = 1 $ case is the implicit midpoint
  method. The $ s = 2 $ and $ s = 3 $ methods have the corresponding
  Butcher tableaux
  \begin{equation*}
    \def\arraystretch{1.5}
    \begin{array}[b]{|cc}
      \frac{1}{4} & \frac{1}{4} - \frac{\sqrt{3}}{6} \\
      \frac{1}{4} + \frac{\sqrt{3}}{6} & \frac{1}{4} \\
      \hline
      \frac{1}{2} & \frac{1}{2}
    \end{array} \quad , \qquad
    \begin{array}[b]{|ccc}
      \frac{5}{36} & \frac{2}{9} - \frac{\sqrt{15}}{15} & \frac{5}{36} - \frac{\sqrt{15}}{30} \\
      \frac{5}{36} + \frac{\sqrt{15}}{24} & \frac{2}{9} & \frac{5}{36} - \frac{\sqrt{15}}{24} \\
      \frac{5}{36} + \frac{\sqrt{15}}{30} & \frac{2}{9} + \frac{\sqrt{15}}{15} & \frac{5}{36} \\
      \hline
      \frac{5}{18} & \frac{4}{9} & \frac{5}{18}
    \end{array} \quad .
  \end{equation*}
  These methods are implicit but not diagonally implicit, except when
  $ s = 1 $.
\end{example}

\begin{example}[SyDIRK methods]
  \label{ex:sydirk}
  Given nonzero $ b _1, \ldots, b _s $, there is a unique diagonally
  implicit method, called a \emph{symplectic diagonally implicit
    Runge--Kutta} (SyDIRK) method, satisfying
  \eqref{e:rk_coeffs_symplectic}, whose coefficients are given by
    \begin{equation}
    \label{e:sydirk_coeffs}
    a _{ i j } =
    \begin{cases}
      b _j , & j < i , \\
      b _i / 2 , & j = i ,\\
      0 , & j > i .
    \end{cases}
  \end{equation}
  The $ s = 1 $ case is the implicit midpoint method, and it can be
  seen that any $s$-stage SyDIRK method is a composition of implicit
  midpoint steps with step sizes $ h b _1, \ldots, h b _s $. That is,
  if $ \Phi _h $ is the map $ y _0 \mapsto y _1 $ corresponding to the
  implicit midpoint method with step size $h$, then an $s$-stage
  SyDIRK method corresponds to the map
  $ \Psi _h = \Phi _{ h b _s } \circ \cdots \circ \Phi _{ h b _1 }
  $. Higher-order SyDIRK methods may therefore be obtained by applying
  the order theory for symmetric compositions of symmetric methods;
  cf.~\citet{MuSa1999}. The order conditions are substantially simpler
  than those for arbitrary Runge--Kutta methods, which involve
  $ a _{ i j } $ and $ b _i $. The special form
  \eqref{e:sydirk_coeffs} ensures that the order conditions only
  involve $ b _i $, and further simplifications result from the
  symmetry assumption $ b _i = b _{ s - i } $.

  In particular, the order-$4$ condition
  $ \sum _{ i = 1 } ^s b _i ^3 = 0 $ implies that at least one of the
  $ b _i $ must be negative to attain higher order than the implicit
  midpoint method. This is in contrast with the methods of
  \cref{ex:gauss}, where the $ b _i $ are the necessarily-positive
  Gauss--Legendre quadrature weights on $ [ 0, 1 ]$. Consequently,
  SyDIRK methods do not have the strong stability properties for stiff
  systems (e.g., algebraic stability) of Gauss--Legendre collocation
  and other symplectic methods with nonnegative $ b _i $;
  \eqref{e:rk_coeffs_symplectic} alone gives a weaker notion of
  stability that \citet{Cooper1987} calls ``orbital stability,'' which
  is precisely conservation of quadratic invariants. Despite this
  stability trade-off, higher-order SyDIRK and other symplectic
  composition methods remain widely used in geometric numerical
  integration.

  The minimal-stage order-$4$ SyDIRK method with symmetric $ b _i $ is
  the $ s = 3 $ method of \citet{Yoshida1990} and \citet{Suzuki1990},
  \begin{equation*}
    b _1 = b _3 = \frac{ 1 }{ 2 - \sqrt[3]{2} }, \qquad b _2 = - \frac{ \sqrt[3]{2} }{ 2 - \sqrt[3]{2} },
  \end{equation*}
  Interestingly, with SyDIRK and other composition methods, one can
  often improve performance by taking more than the minimum number of
  stages at a given order, allowing the coefficients to be tuned to
  lower the error constant. A well-known example is the order-$4$
  method of \citet{Suzuki1990} with $ s = 5 $ and coefficients
  \begin{equation*}
    b _1 = b _2 = b _4 = b _5 = \frac{ 1 }{ 4 - \sqrt[3]{4}}, \qquad b _3 = - \frac{ \sqrt[3]{4} }{ 4 - \sqrt[3]{4} }.
  \end{equation*}
  See the survey in \citet[Section V.3.2]{HaLuWa2006} and references
  therein.
\end{example}

\begin{remark}
  For SyDIRK methods, there is no loss of generality in assuming that
  $ b _1 , \ldots, b _s $ are all nonzero: if $ b _i = 0 $, then
  \eqref{e:rk_coeffs_symplectic} implies $ b _j a _{ j i } = 0 $ for
  all $j$, so stage $i$ has no affect on the numerical solution and
  may be removed. See Theorem~VI.4.4 and the remark following its
  proof in \citet{HaLuWa2006}.
\end{remark}

\subsection{Quadratic projectable methods}
\label{sec:qprk}

We now turn our attention to the main question: which of these methods
allow us to express \eqref{e:rk_F} entirely in terms of $ z = F (y) $
when $f$ is $F$-projectable? Recall that $f$ is $F$-projectable when
there exists a vector field $g$ on $Z$ such that
$ F ^\prime (y) f (y) = g \bigl( F (y) \bigr) $, i.e., $f$ and $g$ are
$F$-related. For a quadratic functionally equivariant Runge--Kutta
method, we may then write \eqref{e:rk_F_step} as
\begin{equation*}
  z _1 = z _0 + h \sum _{ i = 1 } ^s b _i g ( Z _i ) ,
\end{equation*}
where $ z _1 \coloneqq F ( y _1 ) $, $ z _0 \coloneqq F ( y _0 ) $,
and $ Z _i \coloneqq F ( Y _i ) $. It remains only to consider when we
can compute the internal stages $ Z _i $ entirely on $Z$, without
first computing the stages $ Y _i $.

As an initial idea in this direction, one might think to consider
methods where
\begin{equation*}
  a _{ i j } a _{ i k } - a _{ i j } a _{ j k } - a _{ i k } a _{ k j } = 0 , \qquad i, j, k = 1 , \ldots, s ,
\end{equation*}
which would cause the quadratic terms of \eqref{e:rk_F_stages} to
vanish, just as \eqref{e:rk_coeffs_symplectic} does for
\eqref{e:rk_F_step}. However, this is too much to hope for: taking
$ i = j = k $ shows that $ a _{ i i } = 0 $ for all $i$; taking
$ j = k $ then shows that $ a _{ i j } = 0 $ for all $ i, j $; and
taking $ i = j $ in \eqref{e:rk_coeffs_symplectic} finally shows that
$ b _i = 0 $ for all $i$. Hence, this condition is too strong to be
satisfied by any nontrivial Runge--Kutta method. Instead, we consider
methods satisfying the following slightly weaker condition.

\begin{definition}
  A Runge--Kutta method is \emph{quadratic projectable} if it
  satisfies \eqref{e:rk_coeffs_symplectic} and
  \begin{equation}
  \label{e:rk_coeffs_projectable}
    a _{ i j } a _{ i k } - a _{ i j } a _{ j k } - a _{ i k } a _{ k j } = 0 , \qquad i, j, k = 1 , \ldots, s \quad (j \neq k) .
  \end{equation}
\end{definition}

For a quadratic projectable Runge--Kutta method, we may then write
\eqref{e:rk_F_stages} as
\begin{align*}
  Z _i
  &= z _0 + h \sum _{ j = 1 } ^s a _{ i j } g ( Z _j ) + \frac{ h ^2 }{ 2 } \sum _{ j = 1 } ^s ( a _{ i j } ^2 - 2 a _{ i j } a _{ j j } ) F ^{ \prime \prime } \bigl( f ( Y _j ) , f ( Y _j ) \bigr) \\
  &= z _0 + h \sum _{ j = 1 } ^s a _{ i j }\biggl[  g ( Z _j ) + \frac{ h }{ 2 }  \bigl( a _{ i j } - 2 a _{  j j } \bigr)  F ^{ \prime \prime } \bigl( f ( Y _j ) , f ( Y _j ) \bigr) \biggr] ,
\end{align*}
since only the quadratic terms with $ j = k $ remain in the sum. If we
have the additional assumption (discussed further in
\cref{s:projectable_vector_fields}) that
$ F ^{ \prime \prime } \bigl( f (y) , f (y) \bigr) = \gamma \bigl( F
(y) \bigr) $ for some vector field $\gamma$ on $Z$, then altogether
\begin{equation*}
  Z _i = z _0 + h \sum _{ j = 1 } ^s a _{ i j }\biggl[  g ( Z _j ) + \frac{ h }{ 2 }  \bigl( a _{ i j } - 2 a _{  j j } \bigr) \gamma ( Z _j ) \biggr],
\end{equation*}
and thus the method descends to one on $Z$ involving the vector fields
$g$ and $\gamma$.

We are now ready to prove the main result, establishing the
relationship between quadratic projectable Runge--Kutta methods and
SyDIRK methods, which we do with the aid of the following lemma.

\begin{lemma}
  \label{l:partial_order}
  If \eqref{e:rk_coeffs_projectable} holds, then we have the following
  properties:
  \begin{enumerate}[label=\textup{(\alph*)}]
  \item For all $ i \neq j $, if $ a _{ i j } \neq 0 $, then $ a _{ j i } = 0 $. \label{i:partial_order_antisymmetry}
  \item For all $i, j, k $, if $ a _{ i k } \neq 0 $ and
    $ a _{ k j } \neq 0 $, then $ a _{ i j } \neq 0 $. \label{i:partial_order_transitivity}
  \end{enumerate}
\end{lemma}

\begin{proof}
  Taking $ k = i \neq j $ in \eqref{e:rk_coeffs_projectable} gives
  $ a _{ i j } a _{ j i } = 0 $, which proves
  \ref{i:partial_order_antisymmetry}. Next, if $ j \neq k $, then
  taking $ a _{ i j } = 0 $ in \eqref{e:rk_coeffs_projectable} gives
  $ a _{ i k } a _{ k j } = 0 $, which proves the contrapositive of
  \ref{i:partial_order_transitivity}. Finally, if $ j = k $, then
  \ref{i:partial_order_transitivity} holds trivially.
\end{proof}

\begin{theorem}
  \label{t:qprk_iff_sydirk}
  If a Runge--Kutta method satisfies \eqref{e:rk_coeffs_projectable},
  then it is equivalent (up to permutation of stages) to an ERK or
  DIRK method. Consequently, a Runge--Kutta method is quadratic
  projectable if and only if it is a SyDIRK method.
\end{theorem}

\begin{proof}
  For $ i \neq j $, define $ j \prec i $ whenever
  $ a _{ i j } \neq 0 $. \Cref{l:partial_order} says that $ \prec $ is
  a strict partial order, which may be extended to a total order on
  $ \{ 1, \ldots, s \} $. Permuting the stages to sort them according
  to this total order gives a lower-triangular matrix of
  $ a _{ i j } $ coefficients, which therefore corresponds to an ERK
  or DIRK method. In particular, for a quadratic projectable
  Runge--Kutta method, the symplecticity condition
  \eqref{e:rk_coeffs_symplectic} also holds, so the method is a SyDIRK
  method.

  Conversely, SyDIRK coefficients \eqref{e:sydirk_coeffs} clearly
  satisfy \eqref{e:rk_coeffs_symplectic}. To see that they also
  satisfy \eqref{e:rk_coeffs_projectable}, suppose without loss of
  generality that $ k < j $. Then $ a _{ j k } = b _k $ and
  $ a _{ k j } = 0 $, so
  \begin{equation*}
    a _{ i j } a _{ i k } - a _{ i j } a _{ j k } - a _{ i k } a _{ k j } = a _{ i j } ( a _{ i k } - b _k ) .
  \end{equation*}
  If $ i < j $, then $ a _{ i j } = 0 $, and the expression above
  vanishes. Otherwise, $ k < j \leq i $, so $ a _{ i k } = b _k $, and
  again, the expression above vanishes. Hence,
  \eqref{e:rk_coeffs_projectable} holds, and we conclude that SyDIRK
  methods are quadratic projectable.
\end{proof}

\begin{remark}
  From the discussion of order and stability in \cref{ex:sydirk}, it
  follows that a quadratic projectable Runge--Kutta method with
  nonnegative $ b _i $ has order at most $2$; higher-order methods
  must have at least one negative $ b _i $.
\end{remark}

\begin{corollary}
  \label{c:sydirk_descends}
  Let $f$ be a vector field on $Y$, let $ F \colon Y \rightarrow Z $
  be quadratic, and let $g$ and $\gamma$ be vector fields on $Z$ such
  that $ F ^\prime (y) f (y) = g \bigl( F (y) \bigr) $ and
  $ F ^{ \prime \prime } \bigl( f (y) , f (y) \bigr) = \gamma \bigl( F
  (y) \bigr) $ for all $ y \in Y $. Then an $s$-stage SyDIRK method
  for $ \dot{y} = f (y) $ descends to the following method for
  $ \dot{z} = g (z) $:
  \begin{subequations}
    \label{e:sydirk_reduced}
    \noeqref{e:sydirk_reduced_stages,e:sydirk_reduced_step}
    \begin{align}
      Z _i
      &= z _0 + h \sum _{ j = 1 } ^{i-1} b _j g ( Z _j ) + \frac{ h }{ 2 } b _i g ( Z _i ) - \frac{ h ^2 }{ 8 } b _i ^2 \gamma ( Z _i ) , \label{e:sydirk_reduced_stages} \\
      z _1
      &= z _0 + h \sum _{ i = 1 } ^s b _i g ( Z _i ) . \label{e:sydirk_reduced_step} 
    \end{align}
  \end{subequations}
\end{corollary}

\begin{proof}
  This follows immediately from \cref{l:rk_F} and
  \eqref{e:sydirk_coeffs}.
\end{proof}

\begin{example}[matrix Lie--Poisson reduction]
  \label{ex:lie-poisson}
  Let
  $ Y = \mathbb{R} ^{ n \times n } \times \mathbb{R} ^{ n \times n } $
  and $ Z = \mathbb{R} ^{ n \times n } $, where
  $ \mathbb{R} ^{ n \times n } $ denotes the space of real
  $ n \times n $ matrices, and let $ F ( q, p ) = q ^T p $. Suppose
  $ H \colon Y \rightarrow \mathbb{R} $ is a Hamiltonian of the form
  $ H = \eta \circ F $, i.e., $ H ( q, p ) = \eta ( q ^T p ) $, where
  $ \eta \colon Z \rightarrow \mathbb{R} $ is a reduced
  Hamiltonian. Hamilton's equations on $Y$ can be written as
  \begin{subequations}
    \label{e:matrix_hamiltonian}
    \begin{alignat}{2}
      \dot{q} &= \hphantom{-} \nabla _p H ( q, p ) &&= \hphantom{-} q \nabla \eta ( q ^T p ) ,\\
      \dot{p} &= - \nabla _q H ( q, p ) &&= - p \nabla \eta ( q ^T p ) ^T .
    \end{alignat}
  \end{subequations}
  We then observe that
  \begin{align*}
    \frac{\mathrm{d}}{\mathrm{d}t} ( q ^T p )
    &= \dot{q} ^T p + q ^T \dot{p} \\
    &= \nabla \eta ( q ^T p ) ^T q ^T p - q ^T p \nabla \eta ( q ^T p ) ^T \\
    &= \operatorname{ad} ^\ast _{ \nabla \eta ( q ^T p ) } q ^T p .
  \end{align*}
  In other words, the Hamiltonian vector field on $Y$,
  \begin{equation*}
    f ( q, p ) = \Bigl( q \nabla \eta ( q ^T p ) , - p \nabla \eta ( q ^T p ) ^T  \Bigr) ,
  \end{equation*}
  is $F$-related to the reduced Hamiltonian vector field on $Z$,
  \begin{equation*}
    g (z) = \operatorname{ad} ^\ast _{ \nabla \eta (z) } z ,
  \end{equation*}
  which are the \emph{Lie--Poisson equations} on
  $ Z \cong \mathfrak{gl}(n) ^\ast $. Observe also that
  \begin{equation*}
    F ^{ \prime \prime } \bigl( f (q, p), f ( q, p ) \bigr) = - 2 \nabla \eta ( q ^T p ) ^T q ^T p \nabla \eta ( q ^T p ) ^T \eqqcolon \gamma ( q ^T p ) ,
  \end{equation*}
  where
  $ \gamma (z) \coloneqq - 2 \nabla \eta (z) ^T z \nabla \eta (z) ^T
  $.

  From
  \cref{c:sydirk_descends}, it follows that a SyDIRK method
  for \eqref{e:matrix_hamiltonian} on $Y$ descends to the following
  Lie--Poisson method on $Z$:
  \begin{subequations}
    \label{e:sydirk_lie-poisson}
    \noeqref{e:sydirk_lie-poisson_stages,e:sydirk_lie-poisson_step}
    \begin{align}
      Z _i
      &= z _0 + h \sum _{ j = 1 } ^{i-1} b _j \operatorname{ad} ^\ast _{ \nabla \eta ( Z _j ) } Z _j + \frac{ h }{ 2 } b _i \operatorname{ad} ^\ast _{ \nabla \eta ( Z _i ) } Z _i + \frac{ h ^2 }{ 4 } b _i ^2 \nabla \eta ( Z _i ) ^T Z _i \nabla \eta ( Z _i ) ^T , \label{e:sydirk_lie-poisson_stages}\\
      z _1
      &= z _0 + h \sum _{ i = 1 } ^s b _i \operatorname{ad} ^\ast _{ \nabla \eta ( Z _i ) } Z _i \label{e:sydirk_lie-poisson_step}.
    \end{align}
  \end{subequations}
  For isospectral Hamiltonian systems, this coincides with the
  Lie--Poisson integrators obtained via reduction of SyDIRK methods by
  \citet{DaSLes2022}, who give an alternative expression of
  \eqref{e:sydirk_lie-poisson} in terms of the Cayley transform. To
  see the equivalence with the formulation in \citep{DaSLes2022},
  define the additional intermediate stages
  \begin{equation*}
    Z ^r _i \coloneqq z _0 + h \sum _{ j = 1 } ^i b _i \operatorname{ad} ^\ast _{ \nabla \eta ( Z _j ) } Z _j ,
  \end{equation*}
  where $ Z ^r _0 = z _0 $ and $ Z ^r _s = z _1 $. We may then rewrite
  \eqref{e:sydirk_lie-poisson} as
  \begin{alignat*}{2}
    Z ^r _{ i -1 } &= \biggl( I - \frac{ h }{ 2 } b _i \nabla \eta ( Z _i ) ^T \biggr) Z _i \biggl( I + \frac{ h }{ 2 } b _i \nabla \eta ( Z _i ) ^T \biggr) &&= \operatorname{dcay} ^{-1} _{ h b _i \nabla \eta ( Z _i ) ^T } Z _i ,\\
    Z ^r _i &= \biggl( I + \frac{ h }{ 2 } b _i \nabla \eta ( Z _i ) ^T \biggr) Z _i \biggl( I - \frac{ h }{ 2 } b _i \nabla \eta ( Z _i ) ^T \biggr) &&= \operatorname{dcay} ^{-1} _{ -h b _i \nabla \eta ( Z _i ) ^T } Z _i ,
  \end{alignat*}
  which coincides (modulo slight differences in notation) with
  \citep[Equation~3.6]{DaSLes2022}.
\end{example}

We conclude this section with a result on the numerical evolution of
quadratic observables $ G (z) $ by these methods, analogous to
\cref{l:rk_F}. This shows that \eqref{e:sydirk_reduced} is affine
functionally equivariant (which is clear since affine functions in $z$
are quadratic in $y$) but not quadratic functionally equivariant, in
general, due to the additional $ \mathcal{O} ( h ^3 ) $ term.

\begin{theorem}
  \label{t:quadratic_reduced}
  Under the hypotheses of \cref{c:sydirk_descends}, if
  $ G = G (z) $ is quadratic, then
  \begin{equation*}
    G ( z _1 ) = G ( z _0 ) + h \sum _{ i = 1 } ^s b _i G ^\prime ( Z _i ) g ( Z _i ) + \frac{ h ^3 }{ 8 } \sum _{ i = 1 } ^s b _i ^3 G ^{ \prime \prime } \bigl( g ( Z _i ) , \gamma ( Z _i ) \bigr) .
  \end{equation*}
\end{theorem}

\begin{proof}
  As in the proof of \cref{l:rk_F}, we can use
  \eqref{e:sydirk_reduced_step} to expand
  \begin{equation*}
    G ( z _1 ) = G (z _0 ) + h \sum _{ i = 1 } ^s b _i G ^\prime ( z _0 ) g ( Z _i ) + \frac{ h ^2 }{ 2 } \sum _{ i, j = 1 } ^s b _i b _j G ^{ \prime \prime } \bigl( g ( Z _i ) , g ( Z _j ) \bigr) .
  \end{equation*}
  Next, we use \eqref{e:sydirk_reduced_stages} to expand
  $ G ^\prime ( Z _i ) $ around $ z _0 $, obtaining
  \begin{equation*}
    G ^\prime ( Z _i ) v = G ^\prime ( z _0 ) v + h \sum _{ j = 1 } ^{ i -1 } b _j G ^{ \prime \prime } \bigl(  v , g ( Z _j ) \bigr) + \frac{ h }{ 2 } b _i G ^{ \prime \prime } \bigl( v , g ( Z _i ) \bigr) - \frac{ h ^2 }{ 8 } b _i ^2 G ^{ \prime \prime } \bigl( v , \gamma ( Z _i ) \bigr) ,
  \end{equation*}
  for all $ v \in Z $. Therefore, taking $ v = g ( Z _i ) $ implies
  \begin{multline*}
    h \sum _{ i = 1 } ^s b _i G ^\prime ( z _0 ) g ( Z _i ) = h \sum _{ i = 1 } ^s b _i G ^\prime ( Z _i ) g ( Z _i ) - \frac{ h ^2 }{ 2 } \sum _{ i, j = 1 } ^s b _i b _j G ^{ \prime \prime } \bigl( g ( Z _i ) , g ( Z _j ) \bigr) \\
    +  \frac{ h ^3 }{ 8 } \sum _{ i = 1 } ^s b _i ^3 G ^{ \prime \prime } \bigl( g ( Z _i ) , \gamma ( Z _i ) \bigr) .
  \end{multline*}
  Substituting this into the expression for $ G ( z _1 ) $, the double
  sums cancel, yielding the claimed result.
\end{proof}

\section{Quadratic projectable vector fields}
\label{s:projectable_vector_fields}

In this section, we study a family of vector fields that satisfies the
hypotheses of \cref{c:sydirk_descends}, for which SyDIRK
methods therefore descend from $Y$ to $Z$. We begin by introducing a
general setting for $F$-projectability, where an additional algebraic
condition involving Jordan operator algebras allows us to write
$ F ^{ \prime \prime } (y) \bigl( f (y) , f (y) \bigr) = \gamma \bigl(
F (y) \bigr) $. After this, we consider a different algebraic setting
involving Lie algebra actions, including the important case of
momentum maps in Hamiltonian mechanics, and establish similar results
on projectability and descent of SyDIRK methods.

\subsection{Projectability and Jordan operator algebras}
\label{sec:jordan}

We first characterize $F$-projectability with respect to an arbitrary
$ C ^2 $ map $ F \colon Y \rightarrow Z $, with quadratic maps being a
special case. Throughout this section, we denote by $ L ( Y, Y ) $ and
$ L ( Z, Z ) $ the spaces of linear operators on $Y$ and on $Z$,
respectively.

\begin{proposition}
  \label{prop:projectability}
  Given a $ C ^2 $ map $ F \colon Y \rightarrow Z $, let $f$ be a
  $ C ^1 $ vector field on $Y$. Then $f$ is $F$-projectable if and
  only if, for all $ y _0 \in Y $, there exists a continuous map
  $ B \colon Z \rightarrow L ( Z, Z ) $ such that for all $ y \in Y $,
  \begin{equation}
    \label{e:projectability_condition}
    F ^\prime (y) f (y) = F ^\prime ( y _0 ) f ( y _0 ) + B \bigl( F (y) \bigr) \bigl( F (y) - F ( y _0 ) \bigr) .
  \end{equation}
\end{proposition}

\begin{proof}
  $ ( \Rightarrow ) $ Suppose $f$ is $F$-projectable, i.e., there
  exists a $ C ^1 $ vector field $g$ on $Z$ such that
  $ F ^\prime (y) f (y) = g \bigl( F (y) \bigr) $ for all $ y \in Y
  $. Now, for any $ z _0, z \in Z $, applying the fundamental theorem
  of calculus and chain rule along the line segment
  $ t \mapsto ( 1 - t ) z _0 + t z $ gives
  \begin{align*}
    g (z)
    &= g ( z _0 ) + \int _0 ^1 \frac{\mathrm{d}}{\mathrm{d}t} g \bigl( ( 1 - t ) z _0 + t z  \bigr) \,\mathrm{d}t \\
    &= g ( z _0 ) + \biggl[ \int _0 ^1 g ^\prime \bigl( ( 1 - t ) z _0 + t z \bigr) \,\mathrm{d}t \biggr] ( z - z _0 ) ;
  \end{align*}
  see~\citet[Proposition 2.4.7]{AbMaRa1988}. Let
  $ B (z) = \int _0 ^1 g ^\prime \bigl( ( 1 - t ) z _0 + t z \bigr)
  \,\mathrm{d}t \in L ( Z, Z ) $, which allows the above to be written
  as
  \begin{equation*}
    g (z) = g ( z _0 ) + B (z) ( z - z _0 ) .
  \end{equation*}
  (This is essentially the $ C ^1 $ case of Taylor's Theorem applied
  to $g$.) Taking $ z _0 = F ( y _0 ) $ and $ z = F (y) $, then
  applying the $F$-relatedness of $f$ and $g$, yields
  \eqref{e:projectability_condition}.

  $ ( \Leftarrow ) $ Conversely, if \eqref{e:projectability_condition}
  holds, then define $g$ by
  \begin{equation*}
    g (z) = F ^\prime ( y _0 ) f ( y _0 ) + B (z) \bigl( z - F ( y _0 ) \bigr) .
  \end{equation*}
  Then \eqref{e:projectability_condition} immediately implies that
  $ F ^\prime (y) f (y) = g \bigl( F (y) \bigr) $, i.e., $f$ is
  $F$-related to $g$.
\end{proof}

Assuming $ f (0) = 0 $, applying Taylor's Theorem as above with
$ y _0 = 0 $ lets us write
\begin{equation*}
  f (y) = A (y) y,
\end{equation*}
where $ A \colon Y \rightarrow L ( Y , Y ) $. If additionally
$ F (0) = 0 $, then \cref{prop:projectability} implies that
an $F$-related vector field $g$ has the form
\begin{equation*}
  g (z) = B (z) z .
\end{equation*}
We now specialize to the case where $A$ factors through $F$, i.e.,
$ A = \alpha \circ F $ for some
$ \alpha \colon Z \rightarrow L ( Y , Y ) $. In this case, we
establish a sufficient condition for $F$-projectability, with
additional consequences when $\alpha$ takes values in a Jordan
operator algebra on $Y$.

\begin{theorem}
  \label{t:operator_projectable}
  Let $ f (y) = \alpha \bigl( F (y) \bigr) y $, where
  $ \alpha \colon Z \rightarrow \mathcal{A} \subset L ( Y , Y ) $.

  \begin{enumerate}[label=\textup{(\alph*)}]
  \item If there exists
    $ \beta \colon \mathcal{A} \rightarrow L ( Z, Z ) $ such that
    \begin{equation}
      \label{e:beta_condition}
      F ^\prime (y) T y = \beta (T) F (y) ,
    \end{equation}
    for all $ y \in Y $ and $ T \in \mathcal{A} $, then $f$ is
    $F$-related to $ g (z) = \beta \bigl( \alpha (z) \bigr) z
    $. \label{i:beta}

  \item Additionally, if $ \mathcal{A} $ is a Jordan operator algebra
    on $Y$, i.e., a subspace of $ L ( Y, Y ) $ closed under the Jordan
    product $ S \bullet T \coloneqq \frac{1}{2} ( S T + T S ) $, then
    it follows that
    $ F ^{ \prime \prime } (y) \bigl( f (y) , f (y) \bigr) = \gamma
    \bigl( F (y) \bigr) $, where \label{i:gamma}
    \begin{equation*}
      \gamma (z) = \Bigl[ \beta \bigl( \alpha (z) \bigr) ^2 - \beta \bigl( \alpha (z) ^2 \bigr) \Bigr] z .
    \end{equation*}
  \end{enumerate} 
\end{theorem}

\begin{proof}
  From \eqref{e:beta_condition} with
  $ T = \alpha \bigl( F (y) \bigr) $, we immediately get
  \begin{equation*}
    F ^\prime (y) f (y) = F ^\prime (y) \alpha \bigl( F (y) \bigr) y = \beta \Bigl( \alpha \bigl( F (y) \bigr) \Bigr) F (y) = g \bigl( F (y) \bigr),
  \end{equation*}
  which proves \ref{i:beta}. Next, differentiating
  \eqref{e:beta_condition} with respect to $y$ in the direction
  $ T y $ gives
  \begin{equation*}
    F ^{ \prime \prime } (y) ( T y, T y ) + F ^\prime (y) T ^2 y = \beta (T) F ^\prime (y) T y .
  \end{equation*}
  By \eqref{e:beta_condition}, the right-hand side equals
  $ \beta (T) ^2 F (y) $. Furthermore, if $\mathcal{A}$ is closed
  under the Jordan product, then in particular $ T \in \mathcal{A} $
  implies $ T ^2 = T \bullet T \in \mathcal{A} $, so another
  application of \eqref{e:beta_condition} shows that the second term
  on the left-hand side equals $ \beta ( T ^2 ) F (y) $. Altogether,
  \begin{equation*}
    F ^{ \prime \prime } (y) ( T y , T y ) = \bigl[ \beta ( T ) ^2 - \beta ( T ^2 ) \bigr] F (y) ,
  \end{equation*}
  and taking $ T = \alpha \bigl( F (y) \bigr) $ proves \ref{i:gamma}.
\end{proof}

\begin{corollary}
  \label{c:sydirk_alpha_beta}
  Let $ F \colon Y \rightarrow Z $ be quadratic, and let $\alpha$ and
  $\beta$ be as in \cref{t:operator_projectable}, where
  $\mathcal{A}$ is a Jordan operator algebra on $Y$. Then an $s$-stage
  SyDIRK method for $ \dot{y} = \alpha \bigl( F (y) \bigr) y $
  descends to the following method for
  $ \dot{z} = \beta \bigl( \alpha (z) \bigr) z $:
  \begin{align*}
    Z _i
    &= z _0
      \begin{aligned}[t]
        &+ h \sum _{ j = 1 } ^{i-1} b _j \beta \bigl( \alpha  ( Z _j ) \bigr) Z _j + \frac{ h }{ 2 } b _i \beta \bigl( \alpha ( Z _i ) \bigr) Z _i \\
        &- \frac{ h ^2 }{ 8 } b _i ^2 \Bigl[ \beta \bigl( \alpha ( Z _i ) \bigr) ^2 - \beta \bigl( \alpha ( Z _i ) ^2 \bigr) \Bigr] Z _i ,
      \end{aligned}
    \\
    z _1
    &= z _0 + h \sum _{ i = 1 } ^s b _i \beta  \bigl( \alpha ( Z _i ) \bigr) Z _i .
  \end{align*}
\end{corollary}

\begin{proof}
  This follows immediately from \cref{t:operator_projectable}
  together with \cref{c:sydirk_descends}.
\end{proof}

\begin{example}[matrix reduction revisited]
  \label{ex:matrix_revisited}
  As in \cref{ex:lie-poisson}, let
  $ Y = \mathbb{R} ^{ n \times n } \times \mathbb{R} ^{ n \times n } $
  and $ Z = \mathbb{R} ^{ n \times n } $ with $ F ( q, p ) = q ^T p $.
  Let $\mathcal{A}$ consist of linear operators
  $ T _{ M, N } ( q, p ) \coloneqq ( q M, p N ) $ for
  $ M, N \in \mathbb{R} ^{ n \times n } $, i.e., $Y$ acting on itself
  by right multiplication. Then
  \begin{align*}
    F ^\prime ( q, p ) T _{M, N} ( q, p ) = ( q M ) ^T p + q ^T ( p N ) = M ^T (q ^T p) + (q ^T p) N .
  \end{align*}
  Hence, \eqref{e:beta_condition} holds with
  $ \beta ( T _{M, N} ) z = M ^T z + z N $, and furthermore,
  \begin{multline*}
    \bigl[ \beta ( T _{M, N} ) ^2 - \beta ( T _{M, N} ^2 ) \bigr] z \\
    \begin{aligned}
      &= \bigl[ M ^T ( M ^T z + z N ) + ( M ^T z + z N ) N \bigr] - \bigl[ ( M ^2 ) ^T z + z N ^2 \bigr] \\
      &= 2 M ^T z N .
    \end{aligned}
  \end{multline*}
  In particular, when $ N = - M ^T $, we have
  \begin{align*}
    \beta ( T _{M, -M ^T} ) z &= M ^T z - z M ^T = \operatorname{ad} ^\ast _M z ,\\
    \bigl[ \beta ( T _{M, -M^T} ) ^2 - \beta ( T _{M, -M^T} ^2 ) \bigr] z &= - 2 M ^T z M ^T .
  \end{align*}
  This generalizes the reduction procedure in
  \cref{ex:lie-poisson} for the special case
  $ M = \nabla \eta (z) $, i.e.,
  $ \alpha (z) = T _{ \nabla \eta (z), - \nabla \eta (z) ^T } $.
\end{example}

\begin{example}[Lie--Poisson integrators on $\mathbb{R}^3$ via the Hopf map]
  \label{ex:quaternions}
  \citet*{McMoVe2014e} developed \emph{collective Lie--Poisson
    integrators} on $\mathbb{R}^3$ by lifting to $\mathbb{R} ^4 $ and
  then applying symplectic Runge--Kutta methods. We now show how this
  fits into our framework, and how---when the integrator on
  $ \mathbb{R} ^4 $ is a SyDIRK method---one may integrate directly on
  $ \mathbb{R}^3 $ without computing the lifted trajectory on
  $ \mathbb{R}^4 $.
  
  Let $ Y = \mathbb{H} \cong \mathbb{R} ^4 $ and
  $ Z = \operatorname{Im} \mathbb{H} \cong \mathbb{R}^3 $ be the real
  vector spaces of quaternions and pure imaginary quaternions,
  respectively, and define $ F \colon Y \rightarrow Z $ by
  $ F (y) = \frac{ 1 }{ 4 } y k y ^\ast $. (The classical Hopf
  fibration map is $ 4 F $.) Here, the star denotes quaternion
  conjugation, where $ y = y _0 + y _1 i + y _2 j + y _3 k $ has
  conjugate $ y ^\ast = y _0 - y _1 i - y _2 j - y _3 k $. Letting
  $\mathcal{A} $ consist of left-multiplication operators
  $ L _x y \coloneqq x y $, it follows that
  \begin{equation*}
    F ^\prime (y) L _x y = \frac{1}{4} ( x y ) k y ^\ast + \frac{1}{4} y k (x y) ^\ast = x F (y) + F (y) x ^\ast .
  \end{equation*}
  Hence, \eqref{e:beta_condition} holds with
  $ \beta ( L _x ) z = x z + z x ^\ast $. Furthermore,
  \begin{equation}
    \label{e:gamma_quaternions}
    \bigl[ \beta ( L _x ) ^2 - \beta ( L _x ^2 ) \bigr] z = \bigl[ x ( x z + z x ^\ast ) + ( x z + z x ^\ast ) x ^\ast \bigr] - \bigl[ (x x) z + z (x ^\ast x ^\ast) \bigr] = 2 x z x ^\ast .
  \end{equation}
  Now, given a reduced Hamiltonian $\eta$ on $Z$, define
  $ \alpha (z) = L _{ - \frac{1}{2} \nabla \eta (z) } $. Then
  \begin{equation*}
    f (y) = \alpha \bigl( F (y) \bigr) y = - \frac{1}{2} \nabla \eta \bigl( F (y) \bigr) y ,
  \end{equation*}
  is $F$-related to
  \begin{equation*}
    g (z) = \beta \bigl( \alpha (z) \bigr) z = \frac{1}{2} \bigl(  z \nabla \eta (z) - \nabla \eta (z) z \bigr) = z \times \nabla \eta (z) ,
  \end{equation*}
  where the last expression relates the commutator on
  $ \operatorname{Im} \mathbb{H} $ to the cross product on
  $\mathbb{R}^3$. By a similar calculation to that in \citep[Appendix
  B]{McMoVe2014e}, $f$ is the Hamiltonian vector field for the
  \emph{collective Hamiltonian} $ H = \eta \circ F $ on $Y$, while $g$
  is the reduced Hamiltonian vector field for $\eta$ on $Z$. Moreover,
  from \eqref{e:gamma_quaternions} with
  $ x = - \frac{1}{2} \nabla \eta (z) $, we get
  \begin{align*}
    \gamma (z) = - \frac{1}{2} \nabla \eta (z) z \nabla \eta (z) = \bigl( z \times \nabla \eta (z) \bigr) \times \nabla \eta (z) + \frac{1}{2} z \bigl\lvert \nabla \eta (z) \bigr\rvert ^2 ,
  \end{align*}
  again relating multiplication in $ \operatorname{Im} \mathbb{H} $ to
  the dot and cross products on $\mathbb{R}^3$.

  Therefore, by \cref{c:sydirk_alpha_beta}, a SyDIRK
  method for the collective Hamiltonian on $\mathbb{R} ^4 $ descends
  to the following Lie--Poisson method for the reduced Hamiltonian on
  $ \mathbb{R} ^3 $:
  \begin{align*}
    Z _i
    &= z _0
      \begin{aligned}[t]
        &+ h \sum _{ j = 1 } ^{i-1} b _j Z _j \times \nabla \eta ( Z _j ) + \frac{ h }{ 2 } b _i Z _i \times \nabla \eta ( Z _i ) \\
        &- \frac{ h ^2 }{ 8 } b _i ^2 \Bigl[ \bigl( Z _i \times \nabla \eta (Z _i) \bigr) \times \nabla \eta (Z _i) + \frac{1}{2} Z _i \bigl\lvert \nabla \eta (Z _i) \bigr\rvert ^2 \Bigr] ,
      \end{aligned}
    \\
    z _1
    &= z _0 + h \sum _{ i = 1 } ^s b _i Z _i \times \nabla \eta ( Z _i ) .
  \end{align*}
\end{example}

\Cref{ex:matrix_revisited,ex:quaternions} are
special cases of a more general one, where $Y$ is itself an algebra
and $\mathcal{A}$ consists of left- or right-multiplication
operators. The following result is contained in \citet[Theorem
3]{Albert1949}, but the proof is sufficiently simple that we include
it here.

\begin{lemma}
  Given an algebra $Y$, let $ L _x $ and $ R _x $ denote left- and
  right-multiplication, respectively, by $ x \in Y $. That is,
  $ L _x y \coloneqq x y $ and $ R _x y \coloneqq y x $ for
  $ y \in Y $.
  \begin{enumerate}[label=\textup{(\alph*)}]
  \item If $Y$ is \emph{left alternative}, i.e.,
    $ ( L _x )^2 = L _{ x ^2 } $ for all $ x \in Y $, then the
    left-multiplication operators
    $ \mathcal{L} (Y) \coloneqq \{ L _x : x \in Y \} $ form a Jordan
    operator algebra.\label{i:left-alternative}

  \item If $Y$ is \emph{right alternative}, i.e.,
    $ ( R _x ) ^2 = R _{ x ^2 } $ for all $ x \in Y $, then the
    right-multiplication operators
    $ \mathcal{R} (Y) \coloneqq \{ R _x : x \in Y \} $ form a Jordan
    operator algebra.\label{i:right-alternative}
  \end{enumerate} 
\end{lemma}

\begin{proof}
  Statements \ref{i:left-alternative} and \ref{i:right-alternative}
  are clearly equivalent, so we prove only \ref{i:right-alternative},
  following \citep{Albert1949}. The right-alternativity condition
  implies $ ( R _{ x \pm y } ) ^2 = R _{ ( x \pm y ) ^2 } $, and
  therefore,
  \begin{equation*}
    \frac{1}{2} ( R _x R _y + R _y R _x ) = \frac{ 1 }{ 4 } \bigl[ ( R _{x + y} ) ^2 - ( R _{ x - y } ) ^2 \bigr] = R _{ \frac{ 1 }{ 4 } [ ( x + y ) ^2 - ( x - y ) ^2 ] } = R _{ \frac{1}{2} ( x y + y x ) } .
  \end{equation*}
  Hence, $ \mathcal{R} (Y) $ is closed under the Jordan product.
\end{proof}

\begin{example}[octonions]
  Let $ Y = \mathbb{O} \cong \mathbb{R} ^8 $ be the algebra of
  octonions, which is (both left and right) alternative but not
  associative. Using the Cayley--Dickson construction, we identify
  $ \mathbb{O} \cong \mathbb{H} \oplus \mathbb{H} $ with product
  $ ( q, p ) ( r, s ) \coloneqq ( q r - s ^\ast p , s q + p r
  ^\ast ) $ and conjugation
  $ ( q, p ) ^\ast \coloneqq ( q ^\ast , - p ) $. Now, define $ F \colon Y \rightarrow Z = \mathbb{R}  $ by
  \begin{equation*}
    F (y) =  \lvert y \rvert ^2 \coloneqq y y ^\ast .
  \end{equation*}
  Letting $\mathcal{A} = \mathcal{L} (Y) $, as in
  \cref{ex:quaternions}, we have
  \begin{equation*}
    F ^\prime (y) L _x y = ( x y ) y ^\ast + y ( x y ) ^\ast = x F (y) + F (y) x ^\ast ,
  \end{equation*}
  where it can be shown that alternativity of the product implies
  $ ( x y ) y ^\ast = x ( y y ^\ast ) $ and
  $ y (y ^\ast x ^\ast ) = ( y y ^\ast ) x ^\ast $. Hence,
  \eqref{e:beta_condition} again holds with
  $ \beta ( L _x ) z = x z + z x ^\ast $. Furthermore, the calculation
  in \eqref{e:gamma_quaternions} is still valid for octonions, by
  alternativity, so we again have
  \begin{equation*}
    \bigl[ \beta ( L _x ) ^2 - \beta ( L _x ^2 ) \bigr] z = 2 x z x ^\ast.
  \end{equation*} 
  Since $ z \in \mathbb{R} $ commutes with every $ x \in \mathbb{O} $,
  we may simply write these as
  \begin{equation*}
    \beta ( L _x ) z = 2 \operatorname{Re} (x) z, \qquad \bigl[ \beta ( L _x ) ^2 - \beta ( L _x ^2 ) \bigr] z = 2 \lvert x \rvert ^2 z  ,
  \end{equation*}
  where $ \operatorname{Re} (x) \coloneqq ( x + x ^\ast ) / 2 $ is the
  real part of $x \in \mathbb{O} $.

  Finally, if we take $ \alpha (z) = L _{ \frac{1}{2} a (z) } $ for
  some $ a \colon Z \rightarrow Y $, then
  $ f (y) = \frac{1}{2} a \bigl( F (y) \bigr) y $ is $F$-related to
  $ g (z) = \operatorname{Re} \bigl( a (z) \bigr) z $, and we have
  $ \gamma (z) = \frac{1}{2} \bigl\lvert a (z) \bigr\rvert ^2 z
  $. Hence, a SyDIRK method for $\dot{y} = f (y) $ on $ \mathbb{O} $
  descends to the following method for $\dot{z} = g (z) $ on
  $\mathbb{R}$:
  \begin{align*}
    Z _i
    &= z _0 + h \sum _{ j = 1 } ^{i-1} b _j \operatorname{Re} \bigl( a (Z _j ) \bigr) Z _j + \frac{ h }{ 2 } b _i \operatorname{Re} \bigl( a (Z _i ) \bigr) Z _i - \frac{ h ^2 }{ 16 } b _i ^2  \bigl\lvert a (Z _i) \bigr\rvert ^2 Z _i, \\
    z _1
    &= z _0 + h \sum _{ i = 1 } ^s b _i  \operatorname{Re} \bigl( a (Z _i) \bigr) Z _i.
  \end{align*}
\end{example}

\subsection{Lie algebra actions with invariant bilinear forms}
\label{sec:lie}

We now consider quadratic observables arising from Lie algebra actions
with invariant bilinear forms, for which we can prove an analog of
\cref{t:operator_projectable} without the Jordan operator
algebra assumption. This is an important setting that includes
momentum maps in Hamiltonian mechanics.

Let $\mathfrak{g}$ be a Lie algebra acting linearly on $Y$. We can
identify this with a subalgebra
$ \mathfrak{g} \subset \mathfrak{gl} (Y) $, inheriting the commutator
bracket $ [ S, T ] = S T - T S $ and acting via $ y \mapsto T y
$. Suppose $\omega$ is a bilinear form on $Y$ invariant under the
$\mathfrak{g}$-action, in the sense that
\begin{equation*}
  \omega ( T x , y ) + \omega ( x , T y ) = 0,
\end{equation*}
for all $ T \in \mathfrak{g} $ and $ x, y \in Y $. For example, if
$\omega$ is a symplectic form, then this condition means that
$ \mathfrak{g} $ is a subalgebra of the symplectic Lie algebra
$ \mathfrak{sp} (Y) $; if $\omega$ is an inner product, then
$\mathfrak{g}$ is a subalgebra of the special orthogonal Lie algebra
$ \mathfrak{so} (Y) $.

Next, define the quadratic map $ F \colon Y \rightarrow \mathfrak{g}^{\ast} $ by
\begin{equation*}
  \bigl\langle F (y), T \bigr\rangle \coloneqq \frac{1}{2} \omega ( T y, y ) .
\end{equation*}
This is called a \emph{momentum map} of the $\mathfrak{g}$-action in
the case where $\omega$ is a symplectic form on $Y$. We now obtain a
projectability result for this particular form of observable.

\begin{theorem}
  \label{t:quadratic_lie}
  If $ \alpha \colon \mathfrak{g}^{\ast} \rightarrow \mathfrak{g} $,
  then $ f (y) = \alpha \bigl( F (y) \bigr) y $ is $F$-related to
  $ g (z) = \beta \bigl( \alpha (z) \bigr) z $, where
    \begin{equation*}
      \bigl\langle \beta (T) z, S \bigr\rangle = \bigl\langle z , [S, T ] \bigr\rangle ,
    \end{equation*}
    i.e.,
    $ g (z) = - \operatorname{ad} _{ \alpha
      (z) } ^\ast z $.  Additionally, if $\mathfrak{g}$ is closed
    under $ ( S, T ) \mapsto U _T S \coloneqq T S T $, then
    $ F ^{ \prime \prime } \bigl( f (y) , f (y) \bigr) = \gamma \bigl(
    F (y) \bigr) $, where
    \begin{equation*}
      \bigl\langle \gamma (z) , S \bigr\rangle = - 2 \bigl\langle z , \alpha (z) S \alpha (z) \bigr\rangle ,
    \end{equation*}
    i.e., $ \gamma (z) = - 2 U _{ \alpha (z) } ^\ast z $.
\end{theorem}

\begin{proof}
  Observe that for all $ S, T \in \mathfrak{g} $,
  \begin{align*} 
    \bigl\langle F ^\prime (y) T y , S \bigr\rangle
    &= \frac{1}{2} \bigl[ \omega ( S T y, y ) + \omega ( S y, T y ) \bigr] \\
    &= \frac{1}{2} \bigl[ \omega ( S T y, y ) - \omega ( T S y, y ) \bigr]\\
    &= \frac{1}{2} \omega \bigl( [ S, T ] y, y \bigr) \\
    &= \bigl\langle F (y) , [S, T] \bigr\rangle \\
    &= \bigl\langle \beta (T) F (y) , S \bigr\rangle .
  \end{align*}
  Hence, $F$-relatedness of $f$ and $g$ follows by taking
  $ T = \alpha \bigl( F (y) \bigr) $. Note that this is essentially a
  version of the argument in
  \hyperref[t:operator_projectable]{\cref*{t:operator_projectable}\ref*{i:beta}}, where the
  calculation above establishes \eqref{e:beta_condition} with
  $ \mathcal{A} = \mathfrak{g} $ and $ Z = \mathfrak{g}^{\ast} $.
  Next, we have
  \begin{equation*}
    \bigl\langle F ^{ \prime \prime } ( T y, T y ) , S \bigr\rangle = \omega ( S T y, T y ) = - \omega ( T S T y, y ) = - 2 \bigl\langle F (y) , TST \bigr\rangle ,
  \end{equation*}
  and taking $ T = \alpha \bigl( F (y) \bigr) $ again completes the
  proof.
\end{proof}

\begin{corollary}
  \label{c:quadratic_lie_sydirk}
  Under the hypotheses of \cref{t:quadratic_lie}, an $s$-stage
  SyDIRK method for $ \dot{y} = \alpha \bigl( F (y) \bigr) y $
  descends to the following method for
  $ \dot{z} = - \operatorname{ad} ^\ast _{ \alpha (z) } z
  $:
  \begin{align*}
    Z _i
    &= z _0 - h \sum _{ j = 1 } ^{i-1} b _j \operatorname{ad} ^\ast _{ \alpha  ( Z _j ) } Z _j - \frac{ h }{ 2 } b _i \operatorname{ad} ^\ast _{ \alpha  ( Z _i ) } Z _i + \frac{ h ^2 }{ 4 } b _i ^2 U _{ \alpha ( Z _i ) } ^\ast Z _i  , \\
    z _1
    &= z _0 - h \sum _{ i = 1 } ^s b _i \operatorname{ad} ^\ast _{ \alpha  ( Z _i ) } Z _i .
  \end{align*}
\end{corollary}

\begin{remark}
  \label{rem:qla}
  Given any bilinear form $\omega$ on $Y$, we could define
  \begin{equation*}
    \mathfrak{g}  = \bigl\{ T \in \mathfrak{gl} (Y) : \omega ( T x, y ) + \omega ( x, T y ) = 0 ,\ \forall x, y \in Y \bigr\} .
  \end{equation*}
  In the geometric numerical integration literature, these are
  sometimes referred to as \emph{quadratic Lie algebras}
  \citep[Section IV.8.3]{HaLuWa2006}. (Some authors only use this
  terminology in cases where $\omega$ is nondegenerate \citep[Section
  6]{IsMuNoZa2000}.) In this case, it is easily seen that
  $ S, T \in \mathfrak{g} $ implies both $ [ S, T ] \in \mathfrak{g} $
  and $ U _T S \in \mathfrak{g} $.

  Although $\mathfrak{g}$ is generally not a Jordan algebra, the
  quadratic representation $ T \mapsto U _T $ also arises in the
  theory of quadratic Jordan algebras
  \citep{Jacobson1966,McCrimmon1978}. In particular, when
  $\mathfrak{g}$ is defined from $\omega$ as above, then we have an
  associated Jordan operator algebra
  \begin{equation*}
    \mathcal{A}  = \bigl\{ T \in \mathfrak{gl} (Y) : \omega ( T x, y ) - \omega ( x, T y ) = 0 ,\ \forall x, y \in Y \bigr\} ,
  \end{equation*}
  and
  $ U _T S = 2 T \bullet ( T \bullet S ) - ( T \bullet T ) \bullet S =
  TST $, where again $ \bullet $ is the Jordan product.
\end{remark}

\begin{remark}
  Some of the earlier examples can also be viewed from the
  perspective of Lie algebra actions:
  \begin{itemize}
  \item Consider the cotangent-lifted right action of
    $ \mathfrak{g} = \mathfrak{gl} (n) $ on
    $ Y = \mathfrak{gl} (n) \times \mathfrak{gl} (n) ^\ast $, which we
    can write as $ T _M ( q, p ) \coloneqq ( q M , - p M ^T ) $. This
    preserves the canonical symplectic form $\omega$ on $Y$, and the
    corresponding momentum map is precisely $ F (q, p) = q ^T p $, as
    in \cref{ex:lie-poisson,ex:matrix_revisited}, cf.~\citet[Theorem
    12.1.4]{MaRa1999}.  Since $\mathfrak{g}$ is a \emph{right} action,
    note that
    $ [ T _{ M _1 } , T _{ M _2 } ] = - T _{ [ M _1 , M _2 ] } $,
    which explains why the expression $ \operatorname{ad} ^\ast _M $
    on $ Z = \mathbb{R} ^{ n \times n } $ changes signs to become
    $ - \operatorname{ad} _{ T _M } ^\ast $ in \cref{t:quadratic_lie}.

  \item Consider the left action of
    $ \mathfrak{g} = \operatorname{Im} \mathbb{H} $ on
    $ Y = \mathbb{H} $ defined by
    $ T _z y \coloneqq - \frac{1}{2} z y $. This preserves the
    symplectic form
    $ \omega ( x, y ) \coloneqq \frac{1}{2} ( x k y ^\ast - y k x
    ^\ast ) $, and the corresponding momentum map is
    $ F (y) = \frac{ 1 }{ 4 } y k y ^\ast $, as in
    \cref{ex:quaternions}.
    
  \end{itemize}
\end{remark}

\section{Application to matrix hydrodynamics}
\label{s:hydrodynamics}

Euler's equations of incompressible hydrodynamics are Lie--Poisson
equations on the dual of the infinite-dimensional Lie algebra of
divergence-free vector fields, cf.~\citet{ArKh1998} and references
therein. For hydrodynamics in 2D (e.g., on a sphere), an approach
pioneered by \citet{Zeitlin1991,Zeitlin2004} approximates this
infinite-dimensional Lie algebra by a finite-dimensional matrix Lie
algebra $ \mathfrak{su} (n) $, i.e., skew-Hermitian $ n \times n $
matrices with trace zero. Thus, matrix Lie--Poisson integrators on
$ \mathfrak{su} (n) ^\ast $ can be used for structure-preserving
simulation of 2D hydrodynamics. This is a straightforward application
of \cref{ex:lie-poisson} (modulo replacing $\mathbb{R}$ by
$\mathbb{C}$ and transpose by adjoint) that has previously been
considered by \citet{MoVi2020,DaSLes2022}.

In this section, we discuss the application of the results in this
paper to two generalizations of the above. The first is Lie--Poisson
integration for matrix semidirect product Lie algebras, which was
recently considered by \citet{MoRo2025} and applied to a Zeitlin-type
approximation of magnetohydrodynamics (MHD) on
$ \mathfrak{su} (n) \ltimes \mathfrak{su} (n) ^\ast $. The second is
an application to the 2D Navier--Stokes equations, which contain
dissipation terms and are no longer Lie--Poisson (except in the
inviscid case, when they coincide with Euler's equations), and to
non-conservative matrix flows more generally.

\subsection{Semidirect products and magnetohydrodynamics}
\label{s:MHD}

\citet{MoRo2025} describe a model of matrix MHD given by Lie--Poisson
equations on the dual of the semidirect product Lie algebra
$ \mathfrak{su} (n) \ltimes \mathfrak{su} (n) ^\ast $, and they show
that the Lie--Poisson integrators of \citet{ModViv2020} may be applied
to this model. However, the semidirect product Lie algebra does not
satisfy the assumptions of \citep{ModViv2020} to be quadratic and
reductive, so the fact that the integrators remain Lie--Poisson in
this case is not immediate. Instead, \citet{MoRo2025} prove this by
lifting to the cotangent bundle of the semidirect product Lie
\emph{group}.

We now show that SyDIRK methods descend to Lie--Poisson integrators
for matrix semidirect product Lie algebras more generally, using a
version of the argument in \cref{ex:lie-poisson,ex:matrix_revisited}
extended to complex-valued matrices.

Elements
$ q = ( q _1 , q _2 ) \in \mathfrak{gl} (n, \mathbb{C} ) \ltimes
\mathfrak{gl} ( n , \mathbb{C} ) ^\ast $ and
$ p = ( p _1 , p _2 ) \in \bigl( \mathfrak{gl} (n, \mathbb{C} )
\ltimes \mathfrak{gl} ( n , \mathbb{C} ) ^\ast \bigr) ^\ast $ may be
identified with the $ 2 n \times 2 n $ triangular matrices
\begin{equation*}
  q =
  \begin{bmatrix}
    q _1  & q _2 ^\dagger \\
          & q _1 
  \end{bmatrix}, \qquad p =
  \begin{bmatrix}
    p _2 ^\dagger \\
    p _1 & p _2 ^\dagger 
  \end{bmatrix}. 
\end{equation*}
For matrices of this type, the quadratic map
$ F ( q, p ) = q ^\dagger p $ has the form
\begin{equation*}
  z = \begin{bmatrix}
    q _1 ^\dagger \\
    q _2 & q _1 ^\dagger 
  \end{bmatrix}
  \begin{bmatrix}
    p _2 ^\dagger \\
    p _1 & p _2 ^\dagger 
  \end{bmatrix} =
  \begin{bmatrix}
    ( p _2 q _1 ) ^\dagger \\
    q _1 ^\dagger p _1 + q _2 p _2 ^\dagger & ( p _2 q _1 ) ^\dagger 
  \end{bmatrix} \eqqcolon
  \begin{bmatrix}
    \theta ^\dagger \\
    w & \theta ^\dagger 
  \end{bmatrix},
\end{equation*}
following the notation of \citep{MoRo2025} for the components
$ ( w, \theta ) \in \bigl( \mathfrak{gl} (n, \mathbb{C} ) \ltimes
\mathfrak{gl} ( n , \mathbb{C} ) ^\ast \bigr) ^\ast $.

Next, as in \cref{ex:matrix_revisited}, let
$ T _{ M , N } ( q, p ) \coloneqq ( q M , p N ) $, where
\begin{equation*}
  M =
  \begin{bmatrix}
    M _1 & M _2 ^\dagger \\
         & M _1 
  \end{bmatrix}, \qquad N =
  \begin{bmatrix}
    N _2 ^\dagger \\
    N _1 & N _2 ^\dagger 
  \end{bmatrix} .
\end{equation*}
Then, by the same calculation as in that example, we get
\begin{align*}
  \beta ( T _{ M , N } ) z
  &= M ^\dagger z + z N \\
  &=
  \begin{bmatrix}
    (\theta M _1 + N _2 \theta) ^\dagger \\
    M _1 ^\dagger w + M _2 \theta ^\dagger + \theta ^\dagger N _1 + w N _2 ^\dagger & (\theta M _1 + N _2 \theta) ^\dagger 
  \end{bmatrix}\\
  \bigl[ \beta ( T _{ M , N } ) ^2 - \beta ( T _{ M , N } ^2 ) \bigr] z
  &= 2 M ^\dagger z N \\
  &= 2     \begin{bmatrix}
      ( N _2 \theta M _1 ) ^\dagger\\
      M _1 ^\dagger w N _2 ^\dagger + M _1 ^\dagger \theta ^\dagger N _1 + M _2 \theta ^\dagger N _2 ^\dagger & ( N _2 \theta M _1 ) ^\dagger 
    \end{bmatrix} .
\end{align*}
In particular, when $ N = - M ^\dagger $, we get
\begin{align*}
  \beta ( T _{ M , - M ^\dagger } ) z
  &= [ M ^\dagger , z ] \\
  &=
  \begin{bmatrix}
    [ \theta  , M _1 ] ^\dagger \\
    [ M _1 ^\dagger , w ] + [ M _2 , \theta ^\dagger ] & [ \theta  , M _1 ] ^\dagger
  \end{bmatrix} \\
  \bigl[ \beta ( T _{ M , - M ^\dagger } ) ^2 - \beta ( T _{ M , - M ^\dagger } ^2 ) \bigr] z
  &= - 2 M ^\dagger z M ^\dagger \\
  &=
    -2 \mkern-4mu \begin{bmatrix}
         ( M _1 \theta M _1 ) ^\dagger \\
         M _1 ^\dagger w M _1 ^\dagger + M _1 ^\dagger \theta ^\dagger M _2 + M _2 \theta ^\dagger M _1 ^\dagger & ( M _1 \theta M _1 ) ^\dagger
       \end{bmatrix}\mkern-4mu .
\end{align*}
Now, as in \citet{MoRo2025}, let us suppose that
$ \alpha (z) = T _{ M (z) , - M (z) ^\dagger } $, where
\begin{equation*}
  M (z) =
  \begin{bmatrix}
    M _1 (w) & M _2 (\theta) ^\dagger \\
             & M _1 (w) 
  \end{bmatrix} .
\end{equation*}
Then it follows from \cref{t:operator_projectable} that
\begin{align*}
  g ( w, \theta ) &= \Bigl( \bigl[ M _1 (w) ^\dagger , w \bigr] + \bigl[  M _2 (\theta) , \theta ^\dagger \bigr] , \bigl[  \theta  , M _1 (w) \bigr]  \Bigr)\\
  \gamma ( w, \theta ) &= - 2 \Bigl(
                         \begin{aligned}[t]
                           &M _1 (w) ^\dagger w M _1 (w) ^\dagger + M _1 (w) ^\dagger \theta ^\dagger M _2 (\theta) + M _2 (\theta) \theta ^\dagger M _1 (w) ^\dagger ,\\
                           &M _1 (w) \theta M _1 (w)  \Bigr) .
                         \end{aligned}
\end{align*}
Hence, by \cref{c:sydirk_alpha_beta}, a SyDIRK method in
$ y = ( q, p ) $ descends to the method
\begin{alignat*}{2}
  W _i &= w _0
  &&+ h \sum _{ j = 1 } ^{ i -1 } b _j \Bigl( \bigl[ M _1 ( W _j ) ^\dagger , W _j \bigr] + \bigl[ M _2 ( \Theta _j ) , \Theta _j ^\dagger \bigr] \Bigr) \\
               &&&+ \frac{ h }{ 2 } b _i \Bigl( \bigl[ M _1 ( W _i ) ^\dagger , W _i \bigr] + \bigl[ M _2 ( \Theta _i ) , \Theta _i ^\dagger \bigr] \Bigr) \\
               &&&+ \frac{ h ^2 }{ 4 } b _i ^2 \Bigl( M _1 (W _i) ^\dagger W _i M _1 (W _i) ^\dagger
                   \begin{aligned}[t]
                     &+ M _1 (W _i) ^\dagger \Theta _i ^\dagger M _2 (\Theta _i) \\
                     &+ M _2 (\Theta _i) \Theta _i ^\dagger M _1 (W _i) ^\dagger \Bigr) ,
                   \end{aligned}
  \\
  \Theta _i &= \theta _0
  &&+ h \sum _{ j = 1 } ^{ i -1 } b _j \bigl[ \Theta _j , M _1 ( W _j ) \bigr]  + \frac{ h }{ 2 } b _i \bigl[ \Theta _i , M _1 ( W _i ) \bigr] + \frac{ h ^2 }{ 4 } b _i ^2 M _1 ( W _i ) \Theta _i M _1 ( W _i ) ,\\
  w _1 &= w _0 &&+ h \sum _{ i = 1 } ^s b _i \Bigl( \bigl[ M _1 ( W _i ) ^\dagger , W _i \bigr] + \bigl[ M _2 ( \Theta _i ) , \Theta _i ^\dagger \bigr] \Bigr) , \\
  \theta _1 &= \theta _0 &&+ h \sum _{ i = 1 } ^s b _i \bigl[ \Theta _i , M _1 ( W _i ) \bigr] .
\end{alignat*}
Since $F$ is the momentum map of the cotangent-lifted action
$ M \mapsto T _{ M , - M ^\dagger } $, the fact that the SyDIRK method
is symplectic implies that the descendent method is Lie--Poisson
\citep[Theorem 12.4.1]{MaRa1999}.

In the special case when all the $ n \times n $ matrix blocks are
skew-Hermitian, as they are for the subalgebra
$ \mathfrak{su} (n) \ltimes \mathfrak{su} (n) ^\ast $, the system
$\dot{z} = g (z) $ agrees with \citep[Equation
3.8]{MoRo2025}. Furthermore, the $ s = 1 $ case of the method above
(i.e., the descendent of the implicit midpoint method) coincides with
the method of \citep[Equations 4.6 and 4.15]{MoRo2025}.

\begin{remark}
  In the calculation above, we have reduced along a momentum map
  \begin{equation*}
    \bigl( \mathfrak{gl}(n, \mathbb{C} ) \ltimes \mathfrak{gl} (n, \mathbb{C}) ^\ast \bigr)  \times \bigl(  \mathfrak{gl}(n, \mathbb{C} ) \ltimes \mathfrak{gl} (n, \mathbb{C}) ^\ast \bigr)  ^\ast \rightarrow \bigl(  \mathfrak{gl}(n, \mathbb{C} ) \ltimes \mathfrak{gl} (n, \mathbb{C}) ^\ast \bigr)  ^\ast.
  \end{equation*}
  \citet{MoRo2025} reduce along a different momentum map,
  \begin{align*}
    T ^\ast \bigl( \mathrm{SU} (n) \ltimes \mathfrak{su} (n) ^\ast \bigr) &\rightarrow \bigl(  \mathfrak{su}(n) \ltimes \mathfrak{su} (n) ^\ast \bigr)  ^\ast \\
    ( Q, m, P , \alpha ) &\mapsto \biggl( \frac{ P Q ^\dagger - Q P ^\dagger }{ 2 } , Q \alpha Q ^\dagger \biggr) \eqqcolon ( w ^\dagger , \theta ),
  \end{align*}
  corresponding to a cotangent-lifted left Lie group action. Since $m$
  is irrelevant and $\alpha$ remains constant, we may restrict
  attention to $ ( Q, P ) \in T ^\ast \mathrm{SU} (n) $, where
  $\alpha$ is treated as a parameter of the system \citep[Remarks
  7 and 8]{MoRo2025}.

  It is now possible to apply \cref{t:operator_projectable} and
  \cref{c:sydirk_alpha_beta} in this alternative setting,
  where we embed $ T ^\ast \mathrm{SU} (n) $ into
  $ Y \coloneqq \mathbb{C} ^{ n \times n } \times \mathbb{C} ^{ n
    \times n } $. If we define
  \begin{equation*}
    T _M ( Q, P ) \coloneqq ( M _1 ^\dagger Q , M _1 ^\dagger P + 2 M _2 ^\dagger Q \alpha ^\dagger ) ,
  \end{equation*}
  following \citep[Equation 4.9]{MoRo2025}, a calculation shows that
  \begin{align*}
    \beta ( T _M ) ( w, \theta ) &= ( w M _1 + M _1 ^\dagger w + \theta M _2 - M _2 ^\dagger \theta ^\dagger , M _1 ^\dagger \theta + \theta M _1 ) ,\\
    \bigl[ \beta ( T _M ) ^2 - \beta ( T _M ^2 ) \bigr] ( w, \theta ) &= 2 ( M _1 ^\dagger w M _1 + M _1 ^\dagger \theta M _2 - M _2 ^\dagger \theta ^\dagger M _1 , M _1 ^\dagger \theta M _1 ) .
  \end{align*}
  When all of these matrices are skew-Hermitian, this becomes
  \begin{align*}
    \beta ( T _M ) ( w, \theta ) &= \bigl( [M _1 ^\dagger, w] + [ M _2 , \theta ^\dagger ] , [\theta, M _1] \bigr) ,\\
    \bigl[ \beta ( T _M ) ^2 - \beta ( T _M ^2 ) \bigr] ( w, \theta ) &= -2 ( M _1 ^\dagger w M _1 ^\dagger + M _1 ^\dagger \theta ^\dagger M _2 + M _2  \theta ^\dagger M _1 ^\dagger , M _1 \theta M _1 )  ,
  \end{align*}
  matching the previous expressions and again recovering
  \citep[Equations 4.6 and 4.15]{MoRo2025}.
\end{remark}

\subsection{General matrix flows and 2D Navier--Stokes}

Suppose we wish to integrate an arbitrary matrix flow
$ \dot{z} = g (z) $ on $ Z = \mathbb{C} ^{ n \times n } $, and that we
would like to do so in such a way that agrees with the
previously-discussed Lie--Poisson integrators when
$ g (z) = \bigl[ M (z) ^\dagger , z ] $. In the generic case where $z$
has distinct nonzero eigenvalues $ \lambda _1 , \ldots, \lambda _n $,
we decompose $ g (z) = M (z) ^\dagger z + z N (z) $ as follows.

We first claim that we have a direct sum decomposition
\begin{equation*}
  Z = \mathcal{R} _z \oplus \mathcal{N} _z , 
\end{equation*}
where $ \mathcal{R} _z $ is the range of $ L \mapsto [ L , z ] $ and
$\mathcal{N} _z$ is its nullspace. By the rank-nullity theorem, it
suffices to show that
$ \mathcal{R} _z \cap \mathcal{N} _z = \{ 0 \} $. Note that
$\mathcal{N} _z $ consists of matrices commuting with $z$, i.e., those
that are simultaneously diagonalizable. With respect to the eigenbasis
of $z$, we have
$ [ L , z ] _{ i j } = ( \lambda _j - \lambda _i ) L _{ i j } $, which
clearly vanishes on the diagonal $ i = j $. Thus, $ [ L , z ] $ can
only be diagonal in this basis if it is identically zero, which proves
the claim.

Since $z$ is invertible, multiplication by $z$ is an invertible
operator on $ \mathcal{N} _z $. It follows that any matrix can be
uniquely written as $ [ L , z ] + P z $ for some
$ L \in \mathcal{R} _z $, $ P \in \mathcal{N} _z $.

\begin{remark}
  In the eigenbasis of $z$, this decomposition is simply the splitting
  of a matrix into its off-diagonal and diagonal parts. Explicitly,
  for a matrix $K$ represented in this basis, we have
  \begin{equation}
    L _{ i j } =
    \begin{cases}
      0 , & i = j ,\\
      K _{ i j } / (\lambda _j - \lambda _i ) , & i \neq j ,
    \end{cases}
    \qquad P _{ i i } =
    \begin{cases}
      K _{ i i } / \lambda _i , & i = j ,\\
      0, & i \neq j . \label{e:LP_matrices}
    \end{cases}
  \end{equation}
\end{remark}

Using the decomposition above, we can therefore write
\begin{equation}
  \label{e:LP_splitting}
  g (z) = [ L , z ] + P z .
\end{equation}
This puts us in the situation of \cref{ex:matrix_revisited},
where we have $F$-related vector fields
\begin{equation*}
  f ( q, p ) = \bigl( q M ( q ^\dagger p ) , p N ( q ^\dagger p ) \bigr) , \qquad g (z) = M (z) ^\dagger z + z N (z) ,
\end{equation*}
with $ F ( q, p ) = q ^\dagger p $. To match \eqref{e:LP_splitting}, we take
\begin{equation}
  \label{e:MNLP}
  M (z) ^\dagger = L + \tfrac{1}{2} P , \qquad N (z) = - L + \tfrac{1}{2} P .
\end{equation}
Therefore, \cref{c:sydirk_alpha_beta} implies that a
SyDIRK method for $ \dot{y} = f ( q, p ) $ descends to the following
method for $ \dot{z} = g (z) $:
\begin{subequations}
  \label{e:sydirk_matrix}
  \noeqref{e:sydirk_matrix_stages,e:sydirk_matrix_step}
  \begin{align}
    Z _i
    &= z _0
      + h \sum _{ j = 1 } ^{ i -1 } b _j g ( Z _j ) + \frac{ h }{ 2 } b _i g ( Z _i )  - \frac{ h ^2 }{ 4 } b _i ^2 M ( Z _i ) ^\dagger Z _i N ( Z _i ) , \label{e:sydirk_matrix_stages} \\
    z _1 &= z _0 + h \sum _{ i = 1 } ^s b _i g ( Z _i ) . \label{e:sydirk_matrix_step}
  \end{align}
\end{subequations}
\begin{remark}
  As with the matrix Lie--Poisson methods discussed in
  \cref{ex:lie-poisson}, we can obtain an alternative
  expression in terms of the additional stages
  \begin{equation*}
    Z ^r _i \coloneqq z _0 + h \sum _{ j = 1 } ^i b _i g ( Z _j ) ,
  \end{equation*}
  where $ Z ^r _0 = z _0 $ and $ Z ^r _s = z _1 $. The algorithm
  \eqref{e:sydirk_matrix} may then be written as
  \begin{align*}
    Z ^r _{ i -1 }
    &= \biggl( I - \frac{ h }{ 2 } b _i M ( Z _i ) ^\dagger \biggr) Z _i \biggl( I - \frac{ h }{ 2 } b _i N ( Z _i ) \biggr), \\
    Z ^r _i
    &= \biggl( I + \frac{ h }{ 2 } b _i M ( Z _i ) ^\dagger \biggr) Z _i \biggl( I + \frac{ h }{ 2 } b _i N ( Z _i ) \biggr),
  \end{align*}
  in a similar style to the matrix Lie--Poisson methods of
  \citet{DaSLes2022}.
\end{remark}

To apply this approach to matrix hydrodynamics with dissipation on
$ \mathfrak{su}(n) $, we first consider general matrix flows on
$ \mathfrak{u} (n) $, the real Lie algebra of skew-Hermitian
$ n \times n $ matrices (with not-necessarily-vanishing trace). When
$z \in \mathfrak{u}(n)$, we observe from \eqref{e:LP_matrices} that
the decomposition $ g (z) = [ L , z ] + P z \in \mathfrak{u}(n) $ has
$L$ skew-Hermitian and $P$ Hermitian.  (In fact, since the eigenbasis
of $z$ is unitary, this is an orthogonal decomposition with respect to
the Frobenius inner product, and we recover the ``canonical
splitting'' of \citet{MoVi2020,MoVi2026}.) It then follows from
\eqref{e:MNLP} that $ M = N $, and therefore
$ \gamma (z) = 2 M (z) ^\dagger z M (z) $ is skew-Hermitian as
well. Hence, the method \eqref{e:sydirk_matrix} remains entirely in
$ \mathfrak{u} (n) $. Furthermore, if
$ g \colon \mathfrak{u} (n) \rightarrow \mathfrak{su}(n) $, as in
matrix hydrodynamics, then $ z _0 \in \mathfrak{su}(n) $ implies
$ z _1 \in \mathfrak{su}(n) $, although the internal stages $ Z _i $
may only be in $ \mathfrak{u}(n) $.

\begin{proposition}
  \label{p:frobenius}
  The Frobenius norm of a numerical solution to
  \eqref{e:sydirk_matrix} on $ \mathfrak{u}(n) $ satisfies
  \begin{equation*}
    \frac{1}{2} \lVert z _1 \rVert ^2 = \frac{1}{2} \lVert z _0 \rVert ^2 - h \sum _{ i = 1 } ^s b _i \operatorname{tr} \Bigl( Z _i P ( Z _i ) Z _i \Bigr) - \frac{ h ^3 }{ 4 } \sum _{ i = 1 } ^s b _i ^3 \operatorname{tr} \Bigl( M ( Z _i ) ^\dagger Z _i P ( Z _i ) Z _i M ( Z _i ) \Bigr) .
  \end{equation*}
  In particular, when $ g (z) = \bigl[ M (z) ^\dagger , z \bigr] $, so
  that $ P = 0 $, the Frobenius norm is conserved.
\end{proposition}

\begin{proof}
  We apply \cref{t:quadratic_reduced} with
  $ G (z) = \frac{1}{2} \lVert z \rVert ^2 $. First, observe that
  \begin{align*}
    G ^\prime (z) g (z)
    &= \bigl\langle z, M (z) ^\dagger z + z M (z) \bigr\rangle \\
    &= - \operatorname{tr} \Bigl( z M (z) ^\dagger z + z M (z) z \Bigr) \\
    &= - \operatorname{tr} \Bigl( z P (z) z \Bigr),
  \end{align*}
  since $ M ^\dagger + M = P $. Similarly,
  \begin{align*}
    G ^{ \prime \prime } \bigl( g (z) , \gamma (z) \bigr)
    &= 2 \bigl\langle M  (z) ^\dagger z + z M (z) , M (z) ^\dagger z M (z) \bigr\rangle \\
    &= - 2 \operatorname{tr} \Bigl( M (z) ^\dagger z M (z) ^\dagger z M (z) + M (z) ^\dagger z M (z) z M (z) \Bigr) \\
    &= - 2 \operatorname{tr} \Bigl( M (z) ^\dagger z P (z) z M (z) \Bigr).
  \end{align*}
  Hence, the result follows by \cref{t:quadratic_reduced}.
\end{proof}

\begin{example}[2D Navier--Stokes]
  Let us consider the 2D Navier--Stokes equations in the vorticity formulation
  \begin{equation*}
    \dot{\omega}=\lbrace\Delta^{-1}\omega,\omega\rbrace + \nu \Delta \omega.
  \end{equation*}
  Following \citet{Zeitlin2004}, these equations admit a spatial
  semidiscretization in $ \mathfrak{su}(n) $,
  \begin{equation}
    \label{e:zeitlin_ns}
    \dot{w} = [ \Delta _n ^{-1} w , w ] + \nu \Delta _n w \eqqcolon g (w) .
  \end{equation}
  Here, $ \Delta _n $ is a discrete approximation to the Laplacian
  (e.g., that of \citet{HoYa1998} on the sphere), which is a
  self-adjoint operator on $ \mathfrak{u}(n) $ but not simply an
  $ n \times n $ matrix. The nullspace of $ \Delta _n $ consists of
  imaginary multiples of the identity matrix, so it is invertible on
  $ \mathfrak{su}(n) $. To make sense of ``$ \Delta _n ^{-1} w $''
  for $ w \in \mathfrak{u}(n) $ more generally---as we will need to do
  at the internal stages---we can do one of two equivalent things:
  \begin{itemize}
  \item Take the Moore--Penrose pseudoinverse, i.e.,
    $ \Delta _n ^{-1} \bigl( w - \frac1n ( \operatorname{tr} w ) I \bigr) $.

  \item View $ \Delta _n $ as an invertible operator on the quotient
    $ \mathfrak{pu}(n) = \mathfrak{u}(n)/i \mathbb{R} I $, as in
    \citet[Section 3]{MoVi2026}, and apply $ \Delta _n ^{-1} $ to the
    equivalence class of $w$. The representative of this equivalence
    class in $ \mathfrak{su}(n) $ is exactly
    $ w - \frac1n ( \operatorname{tr} w ) I $, as above.
  \end{itemize} 
  Since the Lie bracket annihilates multiples of the identity, the
  term $ [ \Delta _n ^{-1} w, w ] \in \mathfrak{su}(n) $ is thus
  well-defined for $ w \in \mathfrak{u}(n) $, and we have
  $ g \colon \mathfrak{u}(n) \rightarrow \mathfrak{su}(n) $ as needed.

  It then follows from the preceding discussion that we can apply the
  method \eqref{e:sydirk_matrix} to the system \eqref{e:zeitlin_ns},
  where $L$ is skew-Hermitian, $P$ is Hermitian, and $ M = N $, and
  this method is Lie--Poisson in the inviscid case $ \nu = 0 $.
  \cref{t:quadratic_reduced} and \cref{p:frobenius}
  then give the following result on dissipation of \emph{energy}
  $ -\frac{1}{2} \langle \Delta _n ^{-1} w , w \rangle $ and
  \emph{enstrophy} $ \frac{1}{2} \lVert w \rVert ^2 $.

  \begin{corollary}
    For the method \eqref{e:sydirk_matrix} applied to the
    Navier--Stokes--Zeitlin system \eqref{e:zeitlin_ns}, energy and
    enstrophy are dissipated monotonically up to
    $ \mathcal{O} ( h ^3 ) $ when $ b _i \geq 0 $ for all $i$, and
    enstrophy is conserved exactly in the inviscid case $ \nu = 0 $.
  \end{corollary}

  \begin{proof}
    Calculating $ G ^\prime (w) g (w) $ when $G$ is energy, we get
    \begin{align*}
      - \bigl\langle \Delta ^{-1} _n w , [ \Delta _n ^{-1} w , w ] + \nu \Delta _n w \bigr\rangle
      &= -\nu \langle \Delta ^{-1} _n w, \Delta _n w \rangle \\
      &= -\nu \langle \Delta _n \Delta _n ^{-1} w , w \rangle \\
      &= -\nu \bigl\lVert w - \tfrac1n (\operatorname{tr} w ) I \bigr\rVert ^2 \\
      &\leq 0 ,
    \end{align*}
    where the last equality uses the fact that the composition of
    $ \Delta _n $ with its pseudoinverse is the orthogonal projection
    onto $ \mathfrak{su}(n) $. For enstrophy, we get
    \begin{equation*}
      \bigl\langle w, [ \Delta _n ^{-1} w, w ] + \nu \Delta _n w \bigr\rangle = \nu \langle w, \Delta _n w \rangle \leq 0 ,
    \end{equation*}
    since $ \Delta _n $ is negative-semidefinite. Hence, dissipation
    of both quantities up to $ \mathcal{O} ( h ^3 ) $ follows from
    \cref{t:quadratic_reduced}. In the inviscid case
    $ \nu = 0 $, we have $ P = 0 $, so exact numerical conservation of
    enstrophy follows from \cref{p:frobenius}.
  \end{proof}
\end{example}

\section*{Acknowledgments}

The authors wish to thank Klas Modin for helpful conversations and
generous feedback during the early stages of this project. We also
would like to thank the Isaac Newton Institute for Mathematical
Sciences for support and hospitality during the program ``Geometry,
compatibility and structure preservation in computational differential
equations,'' supported by EPSRC grant number EP/R014604/1, as this
paper grew out of conversations begun there. Ari Stern was supported
by the National Science Foundation under Grant No.~DMS-2208551.

\end{document}